\documentclass[10pt, a4paper]{article} 

\usepackage{amsmath,amssymb,amsfonts} 
\usepackage{amsthm}
\usepackage{amsthm}

\usepackage[latin1]{inputenc}

\theoremstyle{plain}
\newtheorem{thm}{Theorem}[section]
\newtheorem{lem}[thm]{Lemma}

\newtheorem{prop}[thm]{Proposition}
\theoremstyle{definition}
\newtheorem{defi}[thm]{Definition}

\theoremstyle{remark}

\newtheorem{rem}[thm]{Remark}
\newtheorem{com}[thm]{Comment}

\renewenvironment{proof}{\textsc{Proof:}}{\qed}

\title{A BSDE approach to stochastic differential games with incomplete information
  }
 \author{Christine Gr\"un \footnote{Laboratoire de Mathematiques de Brest
UMR 6205, 6 avenue Le Gorgeu CS 93837, 29238 BREST cedex 3, France; email: christine.gruen@univ-brest.fr.
Supported by the Marie Curie Initial Training Network (ITN) project: ``Deterministic and Stochastic Controlled Systems and
Application", FP7-PEOPLE-2007-1-1-ITN, No. 213841-2.}
}

\begin{document}
\maketitle


\begin{abstract}
We consider a two-player zero-sum stochastic differential game in which one of the players has a private information on the game. Both players observe each other, so that the non-informed player can try to guess his missing information. Our aim is to quantify the amount of information the informed player has to reveal in order to play optimally: to do so, we show that the value function of this zero-sum game can be rewritten as a minimization problem over some martingale measures with a payoff given by the solution of a backward stochastic differential equation.
\end{abstract}

\emph{Keywords.} Stochastic Differential Games, Backward Stochastic Differential Equations, Dynamic Programming, Viscosity Solutions\\

\textit{2000 AMS subject classification}: 93E05, 91A05, 90C39, 60G44, 49N70 


\section{Introduction}
In this paper we consider a two player zero-sum game, where the underlying dynamics are given by a diffusion with controlled drift but uncontrolled (non-degenerate) volatility. The game can take place in $I$ different scenarios for the running cost and the terminal outcome as in a classical stochastic differential game. Before the game starts one scenario is picked with the probability $p=(p_{i})_{i\in\{1,\ldots,I\}}\in\Delta(I)$. The information is transmitted only to Player 1. So at the beginning he knows in which scenario he is playing, while Player 2 only knows the probability $p$. It is assumed that both players observe the actions of the other one, so Player 2 might infer from the actions of his opponent in which scenario the game is actually played.

It has been proved in Cardaliaguet and Rainer \cite{CaRa2} that this game has a value. To investigate the game under the perspective of information transmission we establish an alternative representation of this value. We achieve this by directly modeling the amount of information the informed player reveals during the game. To that end we enlarge the canonical Wiener space to a space which carries besides a Brownain motion, c\`adl\`ag martingales with values in $\Delta(I)$. These martingales can be interpreted as possible beliefs of the uninformed player, i.e. the probability in which scenario the game is played in according to his information at time $t$.

The very same ansatz has been used in the case of deterministic differential games in Cardaliaguet and Rainer \cite{CaRa1}, while the original idea of the so called a posteriori martingale can already be found in the classical work of Aumann and Maschler (see \cite{AuMaS}). Bearing in mind the ideas of Hamad\`ene and Lepeltier \cite{HaLe}  we show that the value of our game can be represented by minimizing the solution of a backward stochastic differential equation (BSDE) with respect to possible beliefs of the uninformed player.

A cornerstone in the investigation of stochastic differential games has been laid by Fleming and Souganidis in \cite{FS} who extend the results of Evans and Souganidis \cite{ES} to a stochastic framework. Therein it is shown that under Isaacs condition the value function of a stochastic differential game is given as the unique viscosity solution of a Hamilton-Jacobi-Isaacs (HJI) equation.

The theory of BSDE, which was originally developped by Peng \cite{P} for stochastic control theory, has been introduced to stochastic differential games by Hamad\`ene and Lepeltier \cite{HaLe} and Hamad\`ene, Lepeltier and Peng \cite{HaLePe}. The former results have been extended to cost functionals defined by controlled BSDEs in Buckdahn and Li \cite{BuLi}, where the admissible control processes are allowed to depend on events occurring before the beginning of the game.

The study of games with incomplete information has its starting point in the pioneering work of Aumann and Maschler (see \cite{AuMaS} and references given therein). The extension to stochastic differential games has been given in Cardaliaguet and Rainer \cite{CaRa2}. The proof is accomplished introducing the notion of dual viscosity solutions to the HJI equation of a usual stochastic differential game, where the probability $p$ just appears as an additional parameter. A different unique characterization via the viscosity solution of the HJI equation with an obstacle in the form of a convexity constraint in $p$ is given in Cardaliaguet \cite{Ca}. We use this latter characterization in order to prove our main representation result.

The outline of the paper is as follows. In section 2 we describe the game and restate the results of \cite{CaRa2} and \cite{Ca} which build the basis for our investigation. In section 3 we give our main theorem and derive the optimal behaviour for the informed player under some smoothness condition. The whole section 4 is devoted to the proof of the main theorem, while the appendix provides some proofs of extensions to classical BSDE results, which are necessary for our case.

\section{Setup}

\subsection{Formal description of the game}

Let $\mathcal{C}([0,T];\mathbb{R}^d)$ be the set of continuous functions from $\mathbb{R}$ to $\mathbb{R}^d$, which are constant on $(-\infty,0]$ and on $[T,+\infty)$. We denote by $B_s(\omega_B)=\omega_B(s)$ the coordinate mapping on $\mathcal{C}([0,T];\mathbb{R}^d)$ and define $\mathcal{H}=(\mathcal{H}_s)$ as the filtration generated by $s \mapsto B_s$. We denote $\Omega_t=\{\omega\in\mathcal{C}([t,T];\mathbb{R}^d)$ and $\mathcal{H}_{t,s}$ the $\sigma$-algebra generated by paths up to time $s$ in $\Omega_t$.  Furthermore we provide $\mathcal{C}([0,T];\mathbb{R}^d)$ with the Wiener measure $\mathbb{P}^0$ on $(\mathcal{H}_s)$. 

In the following we investigate a two-player zero-sum differential game starting at a time $t\geq0$ with terminal time $T$. The dynamics are given by a controlled diffusion on $(\mathcal{C}([t,T];\mathbb{R}^d),(\mathcal{H}_{t,s})_{s\in[t,T]}, \mathcal{H},\mathbb{P}^0)$, i.e. for $t\in[0,T], x\in\mathbb{R}^d$
\begin{eqnarray}
 dX^{t,x,u,v}_s=b(s,X^{t,x,u,v}_s,u_s,v_s)ds+\sigma(s,X^{t,x,u,v}_s)dB_s\ \ \ \ 
 X^{t,x}_{t}=x.
\end{eqnarray}

We assume that the controls of the players $u,v$ can only take their values in some set $U$, $V$ respectively, where $U,V$ are compact subsets of some finite dimensional spaces.

Let $\Delta(I)$ denote the simplex of $\mathbb{R}^I$. The objective to optimize is characterized by
\begin{itemize}
	\item[(i)] running costs: $(l_i)_{i\in\{1,\ldots, I\}}:[0,T]\times\mathbb{R}^d\times U \times V \rightarrow \mathbb{R}$
	\item[(ii)] terminal payoffs: $(g_i)_{i\in\{1,\ldots, I\}}:\mathbb{R}^d\rightarrow\mathbb{R}$,
\end{itemize}
which are chosen with probability $p\in \Delta(I)$ before the game starts. Player 1 chooses his control to minimize, Player 2 chooses his control to maximize the expected payoff. We assume both players observe their opponents control. However Player 1 knows which payoff he maximizes, Player 2 just knows the respective probabilities $p_{i}$ for scenario $i\in\{1,\ldots, I\}$.

The following will be the standing assumption throughout the paper.\\
{\bf Assumption (H)}
\begin{itemize}
	\item[(i)] $b:[0,T]\times\mathbb{R}^d\times U \times V \rightarrow \mathbb{R}^d$ is bounded and continuous in all its variables and Lipschitz continuous with respect to $(t,x)$ uniformly in $(u,v)$.
	\item[(ii)] For $1\leq k,l\leq d$ the function $\sigma_{k,l}:[0,T]\times\mathbb{R}^d \rightarrow \mathbb{R}$ is bounded and Lipschitz continuous with respect to $(t,x)$. For any $(t,x)\in[0,T]\times\mathbb{R}^d$ the matrix $\sigma^*(t,x)$ is non-singular and $(\sigma^*(t,x))^{-1}$ is bounded and Lipschitz continuous with respect to $(t,x)$.
	\item[(iii)] $(l_i)_{i\in I}:[0,T]\times\mathbb{R}^d\times U \times V \rightarrow \mathbb{R}$ is bounded and continuous in all its variables and Lipschitz continuous with respect to $(t,x)$ uniformly in $(u,v)$. $(g_i)_{i\in I}:\mathbb{R}^d \rightarrow \mathbb{R}$ is bounded and uniformly Lipschitz continuous.
	\item[(iv)] Isaacs condition: for all $(t,x,\xi,p)\in[0,T]\times\mathbb{R}^d\times\mathbb{R}^d\times\Delta(I)$ 
	\begin{equation}
	\begin{array}{rcl}
		&&\inf_{u\in U}\sup_{v\in V} \left\{\langle b(t,x,u,v),\xi\rangle+\sum_{i=1}^{I}p_il_i(t,x,u,v)\right\}\\
		\ \\
		&&\ \ \ \ =\sup_{v\in V} \inf_{u\in U}\left\{\langle b(t,x,u,v),\xi\rangle+\sum_{i=1}^{I}p_il_i(t,x,u,v)\right\}=:H(t,x,\xi,p).
	\end{array}
	\end{equation}
\end{itemize}

By assumption (H) the Hamiltonian $H$ is  Lipschitz in $(\xi,p)$ uniformly in $(t,x)$ and Lipschitz in $(t,x)$ with Lipschitz constant $c(1+|\xi|)$, i.e. it holds for all $t,t'\in[0,T]$, $x,x'\in\mathbb{R}^d$, $\xi,\xi'\in\mathbb{R}^d$, $p,p'\in\Delta(I)$
\begin{eqnarray}
&&|H(t,x,\xi,p)|\leq c (1+|\xi|)
\end{eqnarray}
and
\begin{eqnarray}
&&|H(t,x,\xi,p)-H(t',x',\xi',p')|\leq c (1+|\xi|)(|x-x'|+|t-t'|)+c|\xi-\xi'|+ c |p-p'|.
\end{eqnarray}

\subsection{Strategies and value function}

\begin{defi}
For any $t\in[0,T[$ an admissible control $u=(u_s)_{s\in[t,T]}$ for Player 1 is a progressively measurable process with respect to the filtration $(\mathcal{H}_{t,s})_{s\in[t,T]}$ with values in $U$.
The set of admissible controls for Player 1 is denoted by $\mathcal{U}(t)$.\\
The definition for admissible controls $v=(v_s)_{s\in[t,T]}$ for Player 2 is similar. The set of admissible controls for Player 2 is denoted by $\mathcal{V}(t)$.
\end{defi}

In differential games with complete information  as in \cite{FS} it is sufficient, that one player chooses at the beginning an admissible control and the other one chooses the optimal reaction to it. In our case the uniformed player tries to infer from the actions of his opponent in which scenario the game is played and adapts his behavior to his beliefs. Thus a permanent interaction has to be allowed. To this end it is necessary to restrict admissible strategies to have a small delay in time. 

\begin{defi}
A strategy for Player 1 at time $t\in[0,T[$ is a map $\alpha:[t,T]\times\mathcal{C}([t,T];\mathbb{R}^d)\times L^0([t,T];V)\rightarrow U$ which is nonanticipative with delay, i.e. there is $\delta>0$ such that for all $s\in[t,T]$ for any $f,f'\in\mathcal{C}([t,T];\mathbb{R}^d)$ and $g,g'\in L^0([t,T];V)$ it holds: $f=f'$  and $g=g'$ a.e. on $[t,s]$ $\Rightarrow$ $\alpha(\cdot,f,g)=\alpha(\cdot,f',g')$ a.e. on  $[t,s+\delta]$. The set of strategies for Player 1 is denoted by $\mathcal{A}(t)$.\\
The definition of strategies $\beta: [t,T]\times\mathcal{C}([t,T];\mathbb{R}^d)\times L^0([t,T];U)\rightarrow V$ for Player 2 is similar. The set of strategies for Player 2 is denoted by $\mathcal{B}(t)$.
\end{defi}

Next we state a slight modification of Lemma 5.1. \cite{CaRa2}
\begin{lem}
One can associate to each pair of strategies $(\alpha,\beta)\in\mathcal{A}(t)\times\mathcal{B}(t)$ a unique couple of admissible controls $(u,v)\in\mathcal{U}(t)\times\mathcal{V}(t)$, such that for all $\omega\in\mathcal{C}([t,T];\mathbb{R}^d)$
\[\alpha(s,\omega,v(\omega))=u_s(\omega)\ \ \ \ \textnormal{ \textit{and}}\ \ \ \ \beta(s,\omega,u(\omega))=v_s(\omega)\ .\]
\end{lem}
The proof is done via a fixed point argument using the delay property of the strategies.\\

Furthermore it is crucial that the players are allowed to choose their strategies with a certain additional randomness. Intuitively this can be explained by the incentive of the players to hide their information. Thus for the evaluation of a game with incomplete information we introduce random strategies. To this end let $\mathcal{I}$ denote a set of probability spaces which is non trivial and stable by finite product.

\begin{defi}
A random strategy for Player 1 at time $t\in[0,T[$ is a a pair $((\Omega_\alpha,\mathcal{G}_\alpha,\mathbb{P}_\alpha),\alpha)$, where $(\Omega_\alpha,\mathcal{G}_\alpha,\mathbb{P}_\alpha)$ is a probability space in $\mathcal{I}$ and $\alpha: [t,T]\times\Omega_\alpha\times\mathcal{C}([t,T];\mathbb{R}^d)\times L^0([t,T]; V)\rightarrow U$ satisfies
\begin{itemize}
\item[(i)] 
 	$\alpha$ is a measurable function, where $\Omega_\alpha$ is equipped with the $\sigma$-field $\mathcal{G}_\alpha$,
\item[(ii)] 
	there exists $\delta>0$ such that for all $s\in[t,T]$ and for any $f,f'\in\mathcal{C}([t,T];\mathbb{R}^d)$ and $g,g'\in L^0([t,T];V))$ it holds: 
        \center{$f=f'$  and $g=g'$ a.e. on $[t,s]$ $\Rightarrow$ $\alpha(\cdot,f,g)=\alpha(\cdot,f',g')$ a.e. on  $[t,s+\delta]$ for any $\omega\in\Omega_\alpha$.}
\end{itemize}
The set of random strategies for Player 1 is denoted by $\mathcal{A}^r(t)$.\\
The definition of random strategies $((\Omega_\beta,\mathcal{G}_\beta,\mathbb{P}_\beta),\beta)$, where $\beta: [t,T]\times \Omega_\beta \times\mathcal{C}([t,T];\mathbb{R}^d)\times L^0([t,T];U)\rightarrow V$ for Player 2 is similar. The set of random strategies for Player 2 is denoted by $\mathcal{B}^r(t)$.
\end{defi}

\begin{rem} Again one can associate to each couple of random strategies $(\alpha,\beta)\in\mathcal{A}^r(t)\times\mathcal{B}^r(t)$ for any $(\omega_\alpha,\omega_\beta)\in\Omega_\alpha\times\Omega_\beta$ a unique couple of admissible strategies $(u^{\omega_\alpha,\omega_\beta},v^{\omega_\alpha,\omega_\beta})\in\mathcal{U}(t)\times\mathcal{V}(t)$, such that for all $\omega\in\mathcal{C}([t,T];\mathbb{R}^d)$, $s\in[t,T]$
\begin{equation*}
\alpha(s,\omega_\alpha,\omega,v^{\omega_\alpha,\omega_\beta}(\omega))=u^{\omega_\alpha,\omega_\beta}_s(\omega)\ \ \ \  \textnormal{ and }\ \ \ \  \beta(s,\omega_\beta,\omega,u^{\omega_\alpha,\omega_\beta}(\omega))=v^{\omega_\alpha,\omega_\beta}_s(\omega)\ .
\end{equation*}
Furthermore $(\omega_\alpha,\omega_\beta)\rightarrow (u^{\omega_\alpha,\omega_\beta},v^{\omega_\alpha,\omega_\beta})$ is a measurable map, from $\Omega_\alpha\times\Omega_\beta$ equipped with the $\sigma$-field $\mathcal{G}_\alpha\otimes\mathcal{G}_\beta$  to $\mathcal{V}(t)\times\mathcal{U}(t)$ equipped with the  Borel $\sigma$-field associated to the $L^1$-distance.
\end{rem}

For any $(t,x,p)\in[0,T[\times\mathbb{R}^d\times\Delta(I)$, $\bar \alpha \in (\mathcal{A}^r(t))^I$, $\beta\in\mathcal{B}^r(t)$ we set
        		\begin{eqnarray}
        			J(t,x,p,\bar \alpha,\beta)=\sum_{i=1}^I p_i \ \mathbb{E}_{\bar \alpha_i,\beta}\left[\int_0^Tl_i(s,X_s^{t,x,\bar \alpha_i,\beta},  (\bar\alpha_i)_s,\beta_s)ds+g_i(X_T^{t,x,	\bar\alpha_i,\beta})\right],
		\end{eqnarray}
where (5) should be understood in the following way. As in Remark 2.5. we associate to ${\bar \alpha_i},\beta$ for any $(\omega_{\bar \alpha_i},\omega_\beta)\in\Omega_{\bar \alpha_i}\times\Omega_\beta$ the couple of controls $(u^{\omega_{\bar \alpha_i},\omega_\beta},v^{\omega_{\bar \alpha_i},\omega_\beta})$. The process $X^{t,x,\bar \alpha_i,\beta}$ is then defined for any $(\omega_{\bar \alpha_i},\omega_\beta)$ as solution to (1) with the associated controls. Furthermore $\mathbb{E}_{\bar \alpha_i,\beta}$ is the expectation on $\Omega_{\bar \alpha_i}\times\Omega_\beta\times\mathcal{C}([t,T];\mathbb{R}^d)$ with respect to the probability $\mathbb{P}_{\bar \alpha_i}\otimes\mathbb{P}_\beta\otimes\mathbb{P}^0$, where $\mathbb{P}^0$ denotes the Wiener measure on $\mathcal{C}([t,T];\mathbb{R}^d).$\\

Under assumption (H) the existence of the value of the game is proved in a more general setting in \cite{CaRa2}.

        	\begin{thm}
       		 For any $(t,x,p)\in[0,T[\times\mathbb{R}^d\times\Delta(I)$ the value of the game with incomplete information $V(t,x,p)$ is given by
			\begin{equation}
				\begin{array}{rcl}
		 			V(t,x,p) &=& \inf_{\bar \alpha \in (\mathcal{A}^r(t))^I}\sup_{\beta\in \mathcal{B}^r(t)} J(t,x,p,\bar \alpha,\beta)\\
					\ \\
					&=&\sup_{\beta\in \mathcal{B}^r(t)} \inf_{\bar \alpha \in (\mathcal{A}^r(t))^I} J(t,x,p,\bar \alpha,\beta).
				\end{array}
			\end{equation}
		        \end{thm}
\begin{rem}
It is well known (e.g. \cite{CaRa2} Lemma 3.1) that it suffices for the uninformed player to use admissible (non-random) strategies if he plays first. Intuitively since he has no information to hide. So we can use in (6) the easier expression
\begin{equation}
		 			V(t,x,p) = \inf_{\bar \alpha \in (\mathcal{A}^r(t))^I}\sup_{\beta\in \mathcal{B}(t)} J(t,x,p,\bar \alpha,\beta).
\end{equation}

\end {rem}

The existence and uniqueness of the value function $V:[0,T]\times \mathbb{R}^d\times\Delta(I)\rightarrow\mathbb{R}$ is first given \cite{CaRa2} using the concept of dual viscosity solutions to HJI equations. Starting from this a characterization of the value function as solution of an obstacle problem is given in \cite{Ca}.

\begin{thm}
The function $V:[0,T[\times \mathbb{R}^d\times\Delta(I)\rightarrow\mathbb{R}$ is the unique viscosity solution to
\begin{equation}
	\min \left\{ \frac{\partial w} {\partial t}+\frac{1}{2}\textnormal{tr}(\sigma\sigma^*(t,x)D_x^2w)+H(t,x,D_xw,p),\lambda_{\min}\left(p,\frac{\partial ^2 w}{\partial p^2}\right)\right\}=0
\end{equation}
with terminal condition $w(T,x,p)=\sum_{i}p_ig_i(x)$, where for all $p\in\Delta(I)$, $A\in\mathcal{S}^I$
\begin{eqnarray*}
\lambda_{\min}(p,A):=\min_{z\in T_{\Delta(I)(p)}\setminus\{0\}} \frac{\langle Az,z\rangle}{|z|^2}.
\end{eqnarray*}
and  $T_{\Delta(I)(p)}$ denotes the tangent cone to $\Delta(I)$ at $p$, i.e. $T_{\Delta(I)(p)}=\overline{\cup_{\lambda>0}(\Delta(I)-p)/\lambda}$ .
\end{thm}

\begin{rem}
Note that unlike the standard definition of viscosity solutions (see e.g. \cite{CIL}) the subsolution property to (8) is required only on the interior of $\Delta(I)$ while the supersolution property to (8) is required on the whole domain $\Delta(I)$ (see \cite{Ca} and \cite{CaRa1}). This is due to the fact that we actually consider viscosity solutions with a state constraint, namely $p\in\Delta(I)\subsetneq\mathbb{R}^I$. For a concise investigation of such problems we refer to \cite{CaDoL}.
\end{rem}

We do not go into detail about the rather technical proof of Theorem 2.7. in {\cite{Ca}}. However there is an easy intuitive explanation of the convexity constraint, which we give in the following remark.

\begin{rem}
Let $(t,x)\in [0,T]\times\mathbb{R}^d$ be fixed. For any $p_0\in\Delta(I)$ let $\lambda\in(0,1)$, $p_1,p_2\in\Delta(I)$, such that  $p_0=(1-\lambda)p_1+\lambda p_2$.\\
We consider the game in two steps. First the initial distribution for the game with incomplete information $p_1,p_2$ is picked with probability $(1-\lambda),\lambda$. If the outcome is transmitted only to Player 1, the value of this game is $V(t,x,(1-\lambda)p_1+\lambda p_2)=V(t,x,p_0)$.\\
On the other hand we consider the game in which both players are told the outcome of the pick of the initial distribution $p_1,p_2$. The expected outcome of this game is $(1-\lambda)V(t,x,p_1)+ \lambda V(t,x,p_2)$.\\
In the first game the informed player knows more, hence, if we make the rather reasonable assumption that the value of information is positive, we have  $V(t,x,p_0)\leq(1-\lambda)V(t,x,p_1)+ \lambda V(t,x,p_2)$.
\end{rem}

\section{Alternative representation of the value function}

\subsection{Enlargement of the canonical space}

In the following we establish a representation of the value function by enlarging the canonical Wiener space to a space which will carry besides a Brownain motion a new dynamic. We use this additional dynamic to model the incorporation of the private information into the game. More precisely we model the probability in which scenario the game is played in according to the information of the uniformed Player 2.\\

To that end let us denote by $\mathcal{D}([0,T];\Delta(I))$ the set of c\`adl\`ag functions from $\mathbb{R}$ to $\Delta(I)$, which are constant on $(-\infty,0)$ and on $[T , +\infty)$. We denote by $\bold p_s(\omega_p)=\omega_{{p}}(s)$ the coordinate mapping on $\mathcal{D}([0,T];\Delta(I))$ and by $\mathcal{G}=(\mathcal{G}_s)$ the filtration generated by $s \mapsto \bold p_s$.
Furthermore we recall that $\mathcal{C}([0,T];\mathbb{R}^d)$ denotes the set of continuous functions from $\mathbb{R}$ to $\mathbb{R}^d$, which are constant on $(-\infty,0]$ and on $[T,+\infty)$. We denote by $B_s(\omega_B)=\omega_B(s)$ the coordinate mapping on $\mathcal{C}([0,T];\mathbb{R}^d)$ and by $\mathcal{H}=(\mathcal{H}_s)$ the filtration generated by $s \mapsto B_s$.
We equip the product space $\Omega:=\mathcal{D}([0,T];\Delta(I))\times \mathcal{C}([0,T];\mathbb{R}^d)$ with the filtration $\mathcal{F}=\mathcal{G}\otimes\mathcal{H}$.\\

For $0\leq t\leq T$ we denote $\Omega_t=\{\omega\in\mathcal{D}([t,T];\Delta(I))\times\mathcal{C}([t,T];\mathbb{R}^d)\}$ and $\mathcal{F}_{t,s}$ the $\sigma$-algebra generated by paths up to time $s$ in $\Omega_t$. Furthermore we define the space 
\[\Omega_{t,s}=\{\omega\in\mathcal{D}([t,s];\Delta(I))\times\mathcal{C}([t,s];\mathbb{R}^d)\}\]
for $0\leq t\leq s \leq T$. If $r\in]t,T[$ and $\omega\in\Omega_t$ then let
\begin{eqnarray*}
\omega_1=1_{[-\infty,r[}\omega\ \ \ \ \ \ \ \ \ \ 
\omega_2=1_{[r,+\infty]}(\omega-\omega_{r-})
\end{eqnarray*}
and denote $\pi\omega=(\omega_1,\omega_2)$. The map $\pi:\Omega_t\rightarrow\Omega_{t,r}\times\Omega_{r}$ induces the identification
	$\Omega_t=\Omega_{t,r}\times\Omega_{r}$
moreover $\omega=\pi^{-1}(\omega_1,\omega_2)$, where the inverse is defined in an evident way.\\

For any measure $\mathbb{P}$ on $\Omega$, we denote by $\mathbb{E}_\mathbb{P}[\cdot]$ the expectation with respect to $\mathbb{P}$. We equip $\Omega$ with a certain class of measures.

\begin{defi}
Given $p\in\Delta(I )$, $t\in[0,T]$, we denote by $\mathcal{P}(t,p)$ the set of probability measures $\mathbb{P}$ on $\Omega$
such that, under $\mathbb{P}$
\begin{itemize}
\item[(i)]	$\bold p$ is a martingale, such that
	$\bold p_s = p$  $\forall s< t$,
	$\bold p_s \in\{e_i , i = 1, \ldots , I \}$  $\forall s \geq T$   $\mathbb{P}$-a.s. and
 $\bold p_T$ is independent of $(B_s)_{s\in(-\infty,T]},$
 \item[(ii)]	$(B_s)_{s \in [0,T]}$ is a Brownian motion, 
	\item[(iii)]  under  $\mathbb{P}$  the processes $B_s$ and $\bold p_s$ are strongly orthogonal, i.e. $\langle B,{\bold{p}}^c \rangle_s=0$ for all $ s \in [0,T]$,  where $\bold{p}^c$ denotes the continuous part of $\bold p$. 
\end{itemize}
\end{defi}

\begin{com}
Assumption (ii) is naturally given by the Brownian structure of the game, while (iii) is merely imposed for technical reasons. Assumption (i) is motivated as follows. Before the game starts the information of the uninformed player is just the initial distribution $p$. The martingale property, implying $\bold p_t=\mathbb{E}_{\mathbb{P}}[\bold p_T|\mathcal{F}_t]$, is due to the best guess of the uniformed player about the scenario he is in. Finally, at the end of the game the information is revealed hence $\bold p_T \in\{e_i , i = 1, \ldots , I \}$ and  since the scenario is picked before the game starts the outcome $\bold p_T $ is independent of the Brownian motion. 
\end{com}



\subsection{BSDEs for stochastic differential games with incomplete information}

An alternative representation of the value of the game is given in \cite{CaRa1} in a simpler setting by directly minimizing the expectation of the Hamiltonian over a similar class of martingale measures $\mathbb{P}$. In our case the drift of the diffusion is controlled by the players, hence the Hamiltonian (2) depends on the first derivative of the value function and a ``direct" representation is not possible.\\

Inspired by the ideas of \cite{HaLe} we use the theory of BSDE to solve this problem. To that end we introduce the following spaces.
For any $p\in\Delta(I )$, $t\in[0,T]$ and fixed $\mathbb{P}\in \mathcal{P}(t,p)$ we denote by $\mathcal{L}^2_T(\mathbb{P})$ the set of a square integrable $\mathcal{F}_T$-measurable random variables.
We define by $\mathcal{H}^{2}(\mathbb{P})$ the space of all predictable processes $\theta$ such that $\int_0^\cdot \theta_s dB_s$ is a square integrable martingale, i.e. $\mathbb{E}\left[\int_0^T \theta_s^2 ds\right]<\infty$, and $\mathcal{I}^{2}(\mathbb{P})=\left\{\int\theta dB:\theta\in\mathcal{H}^{2}(\mathbb{P})\right\}$.
Furthermore we denote by $\mathcal{M}^2_{0}(\mathbb{P})$ the space of square integrable martingales null at zero. In the following we shall identify any $N\in\mathcal{M}^2_{0}(\mathbb{P})$ with its c\`adl\`ag modification.\\

For all $t\in[0,T]$, $x\in\mathbb{R}^d$ we define the process $X^{t,x}$ by
\begin{eqnarray}
X^{t,x}_s=x \ \ \ \ s<t, \ \ \ \ \ \   X^{t,x}_s=x+\int_t^s \sigma(r,X_r^{t,x})dB_r\ \ \ \  s\geq t.
\end{eqnarray}
Let $p\in\Delta(I)$. We consider for each $\mathbb{P}\in \mathcal{P}(t,p)$ the BSDE
\begin{eqnarray}
Y_s^{t,x,\mathbb{P}}=\langle\bold p_T, g(X_T^{t,x})\rangle+\int_s^T H(r,X_r^{t,x},Z^{t,x,\mathbb{P}}_r,\bold{p}_r)dr-\int_s^T\sigma^*(r,X_r^{t,x})Z^{t,x,\mathbb{P}}_rdB_r-N_T+N_s,
\end{eqnarray}
where $N\in\mathcal{M}^2_0(\mathbb{P})$ is strongly orthogonal to $\mathcal{I}^{2}(\mathbb{P})$.\\

Existence and uniqueness results for the BSDE (10) can be found in more generality in \cite{ElKH}. Our case is much simpler, since the driver does not depend on the jump parts. This significantly simplifies the proofs which we give for the reader's convenience in the appendix. Note in particular that as in the standard case we can establish a comparison principle (Theorem A.4.), which will be crucial in our further calculations.\\ 

\begin{thm}
Under the assumption (H) the BSDE (10) has a unique solution $(Y^{t,x},Z^{t,x},N)\in\mathcal{H}^2(\mathbb{P})\times\mathcal{H}^2(\mathbb{P})\times\mathcal{M}^2_0(\mathbb{P})$ and it holds for any $s\leq T$
\begin{equation*}
Y_s^{t,x,\mathbb{P}}=\mathbb{E}_\mathbb{P}\left[\int_{s}^T H(r,X_r^{t,x},Z^{t,x,\mathbb{P}}_r,\bold p_r)dr+\langle\bold p_T, g(X_T^{t,x})\rangle\big|\mathcal{F}_s\right].
\end{equation*}
\end{thm}
In particular it holds
\begin{equation}
Y_{t-}^{t,x,\mathbb{P}}=\mathbb{E}_\mathbb{P}\left[\int_{t}^T H(r,X_r^{t,x},Z^{t,x,\mathbb{P}}_r,\bold p_r)dr+\langle\bold p_T, g(X_T^{t,x})\rangle\big|\mathcal{F}_{t-}\right].
\end{equation}

Fix $t\in[0,T]$, $x\in\mathbb{R}^d$, $p\in\Delta(I)$.  Note that all $\mathbb{P}\in \mathcal{P}(t,p)$ are equal on $\mathcal{F}_{t-}$, i.e. the distribution of $(B_s,\bold p_s)$ $s\in[0,t[$ is given by $\delta(p)\otimes\mathbb{P}^0$, where $\delta(p)$ is the measure under which $\bold p$ is constant and equal to $p$ and $\mathbb{P}^0$ is a Wiener measure. So we can identify each $\mathbb{P}\in \mathcal{P}(t,p)$ on $\mathcal{F}_{t-}$ with a common probability measure $\mathbb{Q}$ and define 
\begin{eqnarray}
W(t,x,p)=\textnormal{essinf}_{\mathbb{P} \in \mathcal{P}(t,p)} Y^{t,x,\mathbb{P}}_{t-}\ \ \ \ \ \mathbb{Q}\textnormal{-a.s.}
\end{eqnarray}


The aim of this paper is to show the following alternative representation for the value function.

\begin{thm}
For any $(t,x,p)\in[0,T[\times\mathbb{R}^d\times\Delta(I)$ the value of the game with incomplete information $V(t,x,p)$ can be characterized as
\begin{eqnarray}
V(t,x,p)=\textnormal{essinf}_{\mathbb{P} \in \mathcal{P}(t,p)}  Y^{t,x,\mathbb{P}}_{t-}.
 \end{eqnarray}
\end{thm}

We give the proof in the section 4, where we first show that $W(t,x,p)$ is a deterministic function. Then we establish a Dynamic Programming Principle and show that $W(t,x,p)$ is a viscosity solution to (8). Since $V(t,x,p)$ is by Theorem 2.8. uniquely defined as the viscosity solution to (8) the equality is immediate. Before, let us first investigate under smoothness assumptions a possible behavior of an optimal measure and show how the representation is related to the original game.

\subsection{A sufficient condition for optimality}

Next we give a sufficient condition for a $\mathbb{P}\in\mathcal{P}(t,p)$ to be optimal in (13). We assume $V\in\mathcal{C}^{1,2,2}([t,T)\times\mathbb{R}^d\times\Delta(I);\mathbb{R})$ and set
\begin{eqnarray*}
	\mathcal{H}&=&\bigg\{(t,x,p)\in[0,T)\times\mathbb{R}^d\times\Delta(I): \frac{\partial V} {\partial t}+\frac{1}{2}\textnormal{tr}(\sigma\sigma^*(t,x)D^2_xV)+H(t,x,D_xV,p)=0\bigg\}
\end{eqnarray*}
and 
\begin{eqnarray*}
	\mathcal{H}(t,x)=\left\{p\in\Delta(I):\ (t,x,p)\in\mathcal{H}\right\}.
\end{eqnarray*}
In the theory of games with incomplete information the set $\mathcal{H}$ is usually called the non-revealing set. This is due to the fact that on $\mathcal{H}$ the value function fullfills the standard HJI equation, hence the informed player is not ``actively" using his information because the belief of the uniformed player stays unchanged.

\begin{thm}
Let $(t,x,p)\in[0,T)\times\mathbb{R}^d\times\Delta(I)$. We assume $V\in\mathcal{C}^{1,2,2}([t,T)\times\mathbb{R}^d\times\Delta(I);\mathbb{R})$. Let $\bar{\mathbb{P}}\in\mathcal{P}(t,p)$, such that
\begin{itemize}
\item [(i)] $\bold p_s\in\mathcal{H}(s,X^{t,x}_s)$ $\forall s\in[t,T]$  $\bar{\mathbb{P}}$-a.s.,
\item[(ii)] $\bar{\mathbb{P}}$-a.s. it holds $\forall s\in[t,T]$
		\begin{eqnarray*}
			V(s,X^{t,x}_s,\bold p_{s})-V(s,X^{t,x}_s,\bold p_{s-})-\langle \frac{\partial}{\partial p}V(s,X^{t,x}_s,\bold p_{s-}),\bold p_{s}-\bold p_{s-}\rangle=0,
		\end{eqnarray*}
\item[(iii)]  $\bold p$ is under $\bar{\mathbb{P}}\in\mathcal{P}(t,p)$ a purely discontinuous martingale.
\end{itemize}
Then $\bar{\mathbb{P}}$ is optimal for $V(t,x,p)$.
\end{thm}

\begin{rem}
The analysis of the deterministic case in \cite{CaRa1} indicates that  the conditions (i) and (ii) might also be necessary even in the non-smooth case. In fact under certain assumtions the conditions (i)-(iii) of Theorem 3.5. can expected to be necessary and sufficient. (See \cite{CaRa1} Example 4.4.)
\end{rem}

\begin{proof}
By definition $V(T,x,p)=\langle g(x),p\rangle$. Since $V\in\mathcal{C}^{1,2,2}$ and  $\bold p$ is purely discontinuous we have by It\^o's formula and the assumptions (i)-(iii)
\begin{eqnarray*}
\langle g(X^{t,x}_T),\bold p_T\rangle&=&V(T,X^{t,x}_T,\bold p_T)\\
&=&V(s,X^{t,x}_s,\bold p_{s})+\int_s^T\left(\frac{\partial}{\partial t}V(r,X^{t,x}_r,\bold p_r)+\frac{1}{2}tr(\sigma\sigma^*(r,X^{t,x}_r)D^2_x V(s,X^{t,x}_r,\bold p_r)\right)dr\\
&&\ \ \ \ \ \ +\int_s^T \sigma^*(r,X^{t,x}_r) D_x V(r,X^{t,x}_r,\bold p_r) dB_r\\
	&&\ \ \ \ \ \ +\sum_{s\leq r\leq T} V(r,X^{t,x}_r,\bold p_{r})-V(r,X^{t,x}_r,\bold p_{r-})-\langle \frac{\partial}{\partial p}V(r,X^{t,x}_r,\bold p_{r-}),\bold p_{r}-\bold p_{r-}\rangle\\
	&=&V(s,X^{t,x}_s,\bold p_{s})-\int_s^T H(r,X^{t,x}_r,D_xV(r,X^{t,x}_r,\bold p_{r}),\bold p_r)dr+\int_s^T \sigma^*(r,X^{t,x}_r) D_x V(r,X^{t,x}_r,\bold p_r) dB_r.
\end{eqnarray*}
So by comparison (Theorem A.4.) $(Y^{t,x,\bar{\mathbb{P}}}_s,Z^{t,x,\bar{\mathbb{P}}}_s,N^{t,x,\bar{\mathbb{P}}}_s):=(V(s,X^{t,x}_s,\bold p_{s}),D_xV(s,X^{t,x}_s,\bold p_{s}),0)$ is the unique solution to the BSDE (10).\\
We have in particular
\begin{eqnarray*}
V(t,x,p)=\langle g(X^{t,x}_T),\bold p_T\rangle-\int_t^T H(s,X^{t,x}_s,Z^{t,x,\bar{\mathbb{P}}}_s,\bold p_s)ds+\int_t^T \sigma^*(s,X^{t,x}_s) Z^{t,x,\bar{\mathbb{P}}}_s dB_s,
\end{eqnarray*}
hence the result follows from taking conditional expectation and the representation in Theorem 3.4.
\end{proof}

\subsection{Optimal information reveal for the informed player}

Our aim is to quantify the amount of information the informed player has to reveal in order to play optimally. Note that in the representation we consider as in \cite{HaLe} the original game under a Girsanov transformation. Hence an optimal measure in (13) gives an information structure of the game only up to a Girsanov transformation, which we have to reverse to get back to our original problem.\\

We assume $V\in\mathcal{C}^{1,2,2}([t,T)\times\mathbb{R}^d\times\Delta(I);\mathbb{R})$. Let $\bar{\mathbb{P}}\in\mathcal{P}(t,p)$, such that the conditions of Theorem 3.6. are fulfilled, hence $Z^{t,x,\bar{\mathbb{P}}}_s=D_xV(s,X^{t,x}_s,\bold p_s)$.\\

Thanks to Isaacs condition, assumption (H) (iv), one can define the function $u^*(t,x,p,\xi)$ as a Borel measurable selection of $\textnormal {argmin}_{u\in U}\{\max_{v\in V} \langle b(t,x,u,v),\xi\rangle +\sum_{i=1}^{I}p_il_i(t,x,u,v)\}$, hence
\begin{eqnarray}
		H(t,x,\xi,p)=\max_{v\in V}\big\{\langle b(t,x,u^*(t,x,p,\xi),v),\xi\rangle+\sum_{i=1}^{I}p_il_i(t,x,u^*(t,x,p,\xi),v)\big\}.
\end{eqnarray}
We define the process
\begin{equation}
\bar{u}_s= u^*(s,X^{t,x}_s,D_xV(s,X^{t,x}_s,\bold p_s),\bold p_s),
\end{equation}
where by definition $\bar{u}$ is progressively measurable with respect to the filtration $(\mathcal{F}_{t,s})_{s\in[t,T]}$ with values in $U$. In the following we will denote the set of such processes the set of relaxed controls $\bar{\mathcal{U}}(t)$ and the set of progressively measurable processes with respect to the filtration $(\mathcal{F}_{t,s})_{s\in[t,T]}$ with values in $V$ the set of relaxed controls $\bar{\mathcal{V}}(t)$.\\

We consider for each relaxed control $v\in\bar{\mathcal{V}}(t)$ the (F)BSDE 
\begin{equation}
\begin{array}{rcl}
X^{t,x}_s&=&x+\int_t^s \sigma(r,X_r^{t,x})dB_r\\
\ \\
Y^{t,x,\bar {u},v}_{s}&=&\langle\bold p_T, g(X^{t,x}_T)\rangle
+\int_{s}^T\big(\langle\bold{p}_r, l(r,X^{t,x}_r,\bar {{u}}_r,v_r)\rangle+\langle b(r,X^{t,x}_r,\bar {u}_r,v_r),D_xV(r,X^{t,x}_r,\bold p_r)\rangle\big) dr\\
\ \\
&& \ \ \ \ \ \ \ \ \ \ \ \ \ \ \ \ \ \ \ \ -\int_s^T\sigma^*(r,X^{t,x}_r)D_xV(r,X^{t,x}_r,\bold p_r)dB_r-(N_T-N_s).
\end{array}
\end{equation}

\begin{thm}
For any  $v\in\bar{\mathcal{V}}(t)$ we have 
\begin{eqnarray}
		Y^{t,x,\bar {{u}},v}_{t-}\leq Y^{t,x,\bar{\mathbb{P}}}_{t-}=V(t,x,p)\ \ \ \ \ \ \bar{\mathbb{P}}\textnormal{-a.s.},
\end{eqnarray}
\end{thm}

\begin{proof} Since 
\begin{equation*}
\begin{array}{l}
H(r,X^{t,x}_r,D_xV(r,X^{t,x}_r,\bold p_r),\bold p_r)\\
\ \\
=\min_{u\in U} \max_{v\in V} \left\{\langle b(r,X^{t,x}_r,u,v),D_xV(r,X^{t,x}_r,\bold p_r)\rangle +\langle\bold p_r,l(r,X^{t,x}_r,u_r,v)\rangle\right\}\\
\ \\
=\max_{v\in V} \left\{\langle b(r,X^{t,x}_r,\bar {{u}}_r,v),D_xV(r,X^{t,x}_r,\bold p_r)\rangle +\langle\bold p_r,l(r,X^{t,x}_r,\bar {{u}}_r,v)\rangle\right\}\\
\ \\
\geq\langle b(r,X^{t,x}_r,\bar {{u}}_r,v_r),D_xV(r,X^{t,x}_r,\bold p_r)\rangle +\langle\bold p_r,l(r,X^{t,x}_r,\bar {{u}}_r,v_r)\rangle,
\end{array}
\end{equation*}
(17) follows from the comparison Theorem A.4.
\end{proof}

As in \cite{HaLe} we define now for any $v\in\bar{\mathcal{V}}(t)$ the equivalent measure $\bar{\mathbb{P}}^{\bar{{u}},v}=(\Gamma_T^{\bar {{u}},v})\bar{\mathbb{P}}$ with
\begin{eqnarray*}
\Gamma^{\bar{{u}},v}_s=\mathcal{E}\left(\int_t^s b(r,X^{t,x}_r,\bar {{u}}_r,v_r)\sigma^*(r,X^{t,x}_r)^{-1}dB_r\right).
\end{eqnarray*}
for $s\geq t$ and $\Gamma^{\bar{{u}},v}_s=1$  for $s< t$. By Girsanov (see e.g. Theorem III.3.24 \cite{JaShi}) we have the following Lemma.

\begin{lem}
For any $p\in\Delta(I)$, $t\in[0,T]$, $v\in \bar{\mathcal{V}}(t)$, it holds
\begin{itemize}
\item[(i)]	 $X^{t,x}$ is under $\bar{\mathbb{P}}^{\bar{{u}},v}$ a solution to
	\begin{equation}
		\begin{array}{rcl}
		X^{t,x}_s&=&x+\int_t^s b(r,X_r^{t,x},\bar {{u}}_r,v_r)dr+\int_t^s \sigma(r,X_r^{t,x})d\bar B_r,
		\end{array}
		\end{equation}
	where $\bar B$ is a $\bar{\mathbb{P}}^{\bar{{u}},v}$-Brownian motion.
\item[(ii)]	$\bold p$ is a $\bar{\mathbb{P}}^{\bar{{u}},v}$ martingale, such that
	$\bold p_s = p$  $\forall s< t$,
	$\bold p_s \in\{e_i , i = 1, . . . , I \}$  $\forall s \geq T$   $\mathbb{P}^{\bar{{u}},v}$-a.s. and
 $\bold p_T$ is independent of $(B_s)_{s\in(-\infty,T]},$	
	\item[(iii)]  under  $\bar{\mathbb{P}}^{\bar{{u}},v}$  the processes $\bar B_s$ and $\bold p_s$ are strongly orthogonal, i.e. $\langle B,{\bold{p}}^c \rangle_s=0$ for all $ s \in [0,T]$,  where $\bold{p}^c$ denotes the continuous part of $\bold p$. 
\end{itemize}
\end{lem}

For any $\beta\in\mathcal{B}(t)$, i.e. $\beta:[t,T]\times\mathcal{C}([t,T];\mathbb{R}^d)\times L^0([t,T];U)\rightarrow V$ is nonanticipative with delay, we can define the process $\beta(\bar u)_s=\beta(s,\cdot,\bar u_s)$. By definition $\beta(\bar u)$ is a $V$-valued process which is progressively measurable with repect to the filtration $\mathcal{F}_{t,s}$ hence $\beta(\bar u)\in\bar{\mathcal{V}}(t)$. So we can define for any $\beta\in\mathcal{B}(t)$ the measure $\bar{\mathbb{P}}^{\bar u,\beta(\bar u)}$.\\

To take into account that the informed player knows the scenario, we define now for any scenario $i\in\{1,\ldots,I\}$ and for any $\beta\in\mathcal{B}(t)$ a probability measure $\bar{\mathbb{P}}^{\bar{{u}},\beta(\bar u)}_i$ by: 
for all $A\in\mathcal{F}$ it holds
 \[\bar{\mathbb{P}}^{\bar{{u}},\beta(\bar u)}_i[A]= \bar{\mathbb{P}}^{\bar{{u}},\beta(\bar u)}[A|\bold{p}_T=e_i]=\frac{1}{p_i}  \bar{\mathbb{P}}^{\bar{{u}},\beta(\bar u)}[A\cap\{\bold{p}_T=e_i\}],\ \ \textnormal {if } p_i>0,\] 
and $\bar{\mathbb{P}}^{\bar{{u}},\beta(\bar u)}_i[A]= \bar{\mathbb{P}}^{\bar{{u}},v}[A]$. Note that by Lemma 3.9. (ii) $\bar B$ is a  $\bar{\mathbb{P}}^{\bar{{u}},v}_i$-Brownian motion, hence $X^{t,x}$ is under $\bar{\mathbb{P}}^{\bar{{u}},v}_i$ a solution of the SDE (18).

\begin{thm}
For any scenario $i=1,\ldots,I$ and any strategy of the uniformed player $\beta\in{\mathcal{B}}(t)$ the information transmission $\bar{\mathbb{P}}^{\bar{{u}},\beta(\bar u)}_i$ is optimal for the informed player in the sense that for any $\beta\in\mathcal{B}(t)$
\begin{eqnarray}
\sum_{i=1}^I p_i\mathbb{E}_{\bar{ \mathbb{P}}^{\bar {{u}},\beta(\bar u)}_i}\left[ \int_{t}^T l_i(s,X^{t,x}_s,\bar {{u}}_s,\beta(\bar u)_s)ds +g_i(X^{t,x}_T) \right]\leq V(t,x,p).
\end{eqnarray}
\end{thm}

\begin{proof}
By definition we have
\begin{eqnarray*}
&&\sum_{i=1}^I p_i\mathbb{E}_{\bar{ \mathbb{P}}^{\bar {{u}},\beta(\bar u)}_i}\left[ \int_{t}^T l_i(s,X^{t,x}_s,(\bar {{u}})_s,\beta(\bar u)_s)ds +g_i(X^{t,x}_T) \right]\\
&&\ \ \ {=}\sum_{i=1}^I p_i\mathbb{E}_{\bar {\mathbb{P}}^{\bar {{u}},\beta(\bar u)}}\left[ \int_{t}^T l_i(s,X_s^{t,x},\bar {{u}}_s,\beta(\bar u)_s)ds+g_i(X^{t,x}_T) \big| \bold{p}_T=e_i \right].
\end{eqnarray*}
Furthermore 
\begin{eqnarray*}
&&\sum_{i=1}^I p_i\mathbb{E}_{\bar {\mathbb{P}}^{\bar {{u}},\beta(\bar u)}}\left[ \int_{t}^T l_i(s,X_s^{t,x},\bar {{u}}_s,\beta(\bar u)_s)ds+g_i(X^{t,x}_T) \big| \bold{p}_T=e_i \right]\\
&&\ \ \ =\sum_{i=1}^I \bar {\mathbb{P}}^{\bar {{u}},\beta(\bar u)}[\bold {p}_T=e_i]\mathbb{E}_{\bar {\mathbb{P}}^{\bar {{u}},\beta(\bar u)}}\left[ \int_{t}^T l_i(s,X_s^{t,x},\bar {{u}}_s,\beta(\bar u)_s)ds+g_i(X^{t,x}_T) \big| \bold{p}_T=e_i \right]\\
&&\ \ \ =\sum_{i=1}^I \mathbb{E}_{\bar {\mathbb{P}}^{\bar {{u}},\beta(\bar u)}}\left[ 1_{\{\bold {p}_T=e_i\}}\int_{t}^T l_i(s,X_s^{t,x},\bar {{u}}_s,\beta(\bar u)_s)ds+g_i(X^{t,x}_T) \right]\\
&&\ \ \ =\mathbb{E}_{\bar {\mathbb{P}}^{\bar {{u}},\beta(\bar u)}}\left[ \langle \bold{p}_T, \int_{t}^T l(s,X^{t,x}_s,\bar {{u}}_s,\beta(\bar u)_s)ds\rangle +\langle\bold p_T, g(X^{t,x}_T)\rangle\right]\\
&&\ \ \ =\mathbb{E}_{\bar {\mathbb{P}}^{\bar {{u}},\beta(\bar u)}}\left[  \int_{t}^T\langle\bold{p}_s, l(s,X^{t,x}_s,\bar {{u}}_s,\beta(\bar u)_s)\rangle ds +\langle\bold p_T, g(X^{t,x}_T)\rangle\right],
\end{eqnarray*}
where in the last step we used the product rule for the ${\bar {\mathbb{P}}^{\bar {{u}},\beta(\bar u)}}$-martingale $\bold p$ and the adapted finite variation process $\int_{t}^\cdot l(s,X^{t,x}_s,\bar {{u}}_s,\beta(\bar u)_s)ds$.\\
Furthermore we have
\begin{eqnarray*}
\mathbb{E}_{{\bar {\mathbb{P}}^{\bar {{u}},\beta(\bar u)}}}\left[  \int_{t}^T\langle\bold{p}_s, l(s,X^{t,x}_s,\bar {{u}}_s,\beta(\bar u)_s)\rangle ds +\langle\bold p_T, g(X^{t,x}_T)\rangle\right]=\mathbb{E}_{\bar {\mathbb{P}}^{\bar{{u}},\beta(\bar u)}}\left[ Y^{t,x,\bar {{u}},\beta(\bar u)}_{t-}\right],
\end{eqnarray*}
since by Girsanov $Y^{t,x,\bar {{u}},\beta(\bar u)}_{s}$ is under $\bar {\mathbb{P}}^{\bar{{u}},\beta(\bar u)}$ given by
\begin{equation}
	\begin{array}{rcl}
	Y^{t,x,\bar {{u}},\beta(\bar u)}_{s}&=&\langle\bold p_T, g(X^{t,x}_T)\rangle+\int_{s}^T\langle\bold{p}_r, l(r,X^{t,x}_r,\bar {{u}}_r,\beta(\bar u)_r)\rangle dr\\
	\ \\
			&&\ \ -\int_s^T\sigma^*(r,X^{t,x}_r)D_xV(r,X^{t,x}_r,\bold p_r)d\bar B_r-(N_T-N_s).
\end{array}
		\end{equation}
So since by Theorem 3.7. $Y^{t,x,\bar {{u}},\beta(\bar u)}_{t-}\leq V(t,x,p)$ $\bar{\mathbb{P}}$-a.s. and $\bar{\mathbb{P}}$ is equivalent to ${\bar {\mathbb{P}}^{\bar {{u}},\beta(\bar u)}}$, we have
\[\mathbb{E}_{{\bar {\mathbb{P}}^{\bar {{u}},\beta(\bar u)}}}\left[ Y^{t,x,\bar {{u}},\beta(\bar u)}_{t-}\right]\leq V(t,x,p).\]
\end{proof}

\begin{rem}
In the simpler case of \cite{CaRa1} the representation (13) allowed to derive an optimal random control for the informed player in a direct feedback from. Here however there are significant differences. By the Girsanov transformation we have for each $\beta\in\mathcal{B}(t)$ at each time $s\in[t,T]$ an optimal reaction $\bar{u}_s= u^*(s,X^{t,x}_s,D_xV(s,X^{t,x}_s,\bold p_s),\bold p_s)$ of the informed player. It depends on the state of the system, i.e. $X^{t,x}$ under $\bar {\mathbb{P}}^{\bar {{u}},\beta(\bar u)}_i$ and the shifted randomization $\bold p$ under the optimal measure $\bar {\mathbb{P}}^{\bar {{u}},\beta(\bar u)}_i$. Since this shift depends on the strategy $\beta$ of the uniformed player, we do not find a random control but a kind of random strategy for the informed player. Note that this ``strategy" - none of the less giving us a recipe how the informed player can generate the optimal information flow - is in general not of the form required in definition 2.4. To get a classical random strategy it would be necessary to show a certain structure of the optimal measure $\bar{\mathbb{P}}$. In a subsequent paper we show how this can be established for $\epsilon$-optimal measures leading to $\epsilon$-optimal strategies in the sense of definition 2.4.
\end{rem}

\section{Proof of Theorem 3.4.}

\subsection{The function $W(t,x,p)$ and $\epsilon$-optimal strategies}

Recall that we defined $W(t,x,p)$ $\mathbb{Q}\textnormal{-a.s.}$ as $\textnormal{essinf}_{\mathbb{P} \in \mathcal{P}(t,p)} Y^{t,x,\mathbb{P}}_{t-}$, where by definition a random variable $\xi$ is called  $\textnormal{essinf}_{\mathbb{P} \in \mathcal{P}(t,p)} Y^{t,x,\mathbb{P}}_{t-}$, if
\begin{itemize}
\item[(i)] $\xi\leq Y^{t,x,\mathbb{P}}_{t-}$, $\mathbb{Q}$-a.s., for any $\mathbb{P} \in \mathcal{P}(t,p)$
\item[(ii)] if there is another random variable $\eta$ such that $\eta\leq Y^{t,x,\mathbb{P}}_{t-}$, $\mathbb{Q}$-a.s., for any $\mathbb{P} \in \mathcal{P}(t,p)$, then $\eta\leq\xi$, $\mathbb{Q}$-a.s.
\end{itemize}
So by its very definition $W(t,x,p)$ is merely a ${\mathcal{F}_{t-}}$ measurable random field. However we show that it is deterministic and hence a good candidate to represent the deterministic value function $V(t,x,p)$. Our proof is mainly based on the methods in \cite{BuLi}.

\begin{prop}
For any $t\in[0,T]$, $x\in\mathbb{R}^d$, $p\in\Delta(I)$ it holds
\begin{eqnarray}
W(t,x,p)=\mathbb{E}_{\mathbb{Q}}[W(t,x,p)]\ \ \ \ \ \ \ \ {\mathbb{Q}}\textnormal{-a.s.}
\end{eqnarray}
Hence identifying $W(t,x,p)$ with its deterministic version $\mathbb{E}_{\mathbb{Q}}[W(t,x,p)]$ we can consider $W:[0,T]\times\mathbb{R}^d\times\Delta(I)\rightarrow\mathbb{R}$ as a deterministic function.
\end{prop}

To prove that $W(t,x,p)$ is deterministic it suffices to show that it is independent of the $\sigma$-algebra $\sigma(B_s,{s\in[0,t]})$. Since $\bold {p}$ is on $[0,t[$ $\mathbb{Q}$-a.s. a constant the desired result follows.\\
To show the independence of  $\sigma(B_s,{s\in[0,t]})$ we will use as in \cite{BuLi} a perturbation of $\mathcal{C}([0,T];\mathbb{R}^d)$ with certain elements of the Cameron-Martin space.  Let $H$ denote the Cameron-Martin space of all absolutely continuous elements $h\in\mathcal{C}([0,T];\mathbb{R}^d)$, whose 
Radon-Nikodym derivative $\dot{h}$ belongs to $L^2([0, T ];\mathbb{R}^d)$. Denote $H_t=\{h\in H: h(\cdot)=h(\cdot \wedge t)\}$.
For any $h\in H_t$ , we define  for all $(\omega_p,\omega_B)\in\mathcal{D}([0,T];\Delta(I))\times \mathcal{C}([0,T];\mathbb{R}^d)$ the mapping $\tau_h(\omega_p,\omega_B):=(\omega_p,\omega_B+ h)$. Then $\tau_h : \mathcal{D}([0,T];\Delta(I))\times \mathcal{C}([0,T];\mathbb{R}^d) \rightarrow\mathcal{D}([0,T];\Delta(I))\times \mathcal{C}([0,T];\mathbb{R}^d)$ is a $\mathcal{F}-\mathcal{F}$ measurable bijection with $[\tau_h]^{-1}=\tau_{-h}$.

\begin{lem}
For any $h\in H_t$
\begin{eqnarray}
W(t,x,p)\circ{\tau_h}=W(t,x,p).
\end{eqnarray}
\end{lem}

\begin{proof}
Obviously $\tau_h,\tau_h^{-1}:\mathcal{D}([0,T];\Delta(I))\times \mathcal{C}([0,T];\mathbb{R}^d) \rightarrow\mathcal{D}([0,T];\Delta(I))\times \mathcal{C}([0,T];\mathbb{R}^d)$ is $\mathcal{F}_{t}-\mathcal{F}_{t}$ measurable and $(B_{s}-B_{t})\circ \tau_h=(B_{s}-B_{t})$ for all $s\in[t,T]$.\\

\textbf{Step 1}: Observe that $X^{t,x}_s\circ \tau_h=X^{t,x}_s$ for all $s\in[t,T]$. Then $Y^{t,x,\mathbb{P}}\circ \tau_h$ is the solution to the BSDE
\begin{equation}
\begin{array}{rcl}
(Y^{t,x,\mathbb{P}}\circ \tau_h)_s&=&\langle\bold p_T, g(X_T^{t,x})\rangle+\int_s^TH(r,X_r^{t,x},(Z^{t,x,\mathbb{P}}\circ \tau_h)_r,\bold{p}_r)ds\\
\ \\
&&\ \ \ \ \ -\int_s^T\sigma^*(r,X_r^{t,x})(Z^{t,x,\mathbb{P}}\circ \tau_h)_rdB_r - (N\circ \tau_h)_T+(N\circ \tau_h)_s
\end{array}
\end{equation}
which is the original BSDE (10) however under the different $\mathbb{P}\circ[\tau_h]^{-1}$ dynamics for $\bold p$.\\
Furthermore $X^{t,x}_{s\in[t,T]}$ under $\mathbb{P}$ and under $\mathbb{P} \circ [\tau_h]^{-1}$ are by Girsanov $\mathbb{P}$-a.s. equal. So under $\mathbb{P} \circ [\tau_h]^{-1}$ the process $Y^{t,x,\mathbb{P}\circ [\tau_h]^{-1}}$ by Girsanov solves (23). Since the solution of (23) is unique we have in particular
\begin{eqnarray}
	Y^{t,x,\mathbb{P}}_{t-}\circ \tau_h = Y^{t,x,\mathbb{P}\circ[\tau_h]^{-1}}_{t-}.
\end{eqnarray}

\textbf{Step 2}: We claim that
\begin{eqnarray}
\left(\textnormal{essinf}_{\mathbb{P} \in \mathcal{P}(t,p)} Y^{t,x,\mathbb{P}}_{t-}\right)\circ \tau_h
	=\textnormal{essinf}_{\mathbb{P} \in \mathcal{P}(t,p)} \left(Y^{t,x,\mathbb{P}}_{t-}\circ \tau_h\right)\ \ \ \ \ \ \mathbb{Q}\textnormal{-a.s.}
\end{eqnarray}
Observe that the law of $\tau_h$ is given by
\begin{eqnarray}
\mathbb{P}\circ[\tau_h]^{-1}=\exp\left(\int_0^{t} \dot h_s dB_s-\frac{1}{2}\int_0^{t}|\dot h_s|^2 ds\right)\mathbb{P}
\end{eqnarray}
for all measures $\mathbb{P}$ on $\Omega$.
Define $I(t,x,p)=\textnormal{essinf}_{\mathbb{P} \in \mathcal{P}(t,p)} Y^{t,x,\mathbb{P}}_{t-}$. Then $I(t,x,p)\leq Y^{t,x,\mathbb{P}}_{t-}$.
Since $\mathbb{Q}\circ[\tau_h]^{-1}$ is equivalent to $\mathbb{Q}$  on $\mathcal{F}_{t-}$, we have $I(t,x,p)\circ \tau_h \leq Y^{t,x,\mathbb{P}}_{t-}\circ\tau_h$ $\mathbb{Q}$-a.s.\\
Furthermore let $\xi$ be a $\mathcal{F}_{t-}$-measurable random variable, such that $\xi\leq Y^{t,x,\mathbb{P}}_{t-}\circ\tau_h$ $\mathbb{Q}$-a.s. Then $\xi \circ [\tau_h]^{-1} \leq Y^{t,x,\mathbb{P}}_{t-}$  $\mathbb{Q}$-a.s.. hence it holds $\xi \circ [\tau_h]^{-1} \leq I(t,x,p)$, so $\xi \leq I(t,x,p)\circ \tau_h$.\\
 Consequently we have
\begin{eqnarray*}
	I(t,x,p)\circ \tau_h=\textnormal{essinf}_{\mathbb{P} \in \mathcal{P}(t,p)} (Y^{t,x,\mathbb{P}}_{t-}\circ \tau_h).
\end{eqnarray*}

\textbf{Step 3}: Using (24) and (25) we have $\mathbb{Q}$-a.s.
\begin{eqnarray*}
	W(t,x,p)\circ \tau_h&=&(\textnormal{essinf}_{\mathbb{P} \in \mathcal{P}(t,p)} Y^{t,x,\mathbb{P}}_{t-})\circ \tau_h\\
	&=&\textnormal{essinf}_{\mathbb{P} \in \mathcal{P}(t,p)} (Y^{t,x,\mathbb{P}}_{t-}\circ \tau_h)\\
	&=& \textnormal{essinf}_{\mathbb{P}\in \mathcal{P}(t,p)} Y^{t,x,\mathbb{P}\circ[\tau_h]^{-1}}_{t-}.
\end{eqnarray*}

Note that in general $\mathbb{P} \circ [\tau_h]^{-1}\not\in\mathcal{P}(t,p)$, since under $\mathbb{P} \circ [\tau_h]^{-1}$ the process $B$ is no longer a Brownian motion on $[0,t]$.
We define $\mathbb{P}^h$ on $\Omega=\Omega_{0,t}\times\Omega_t$, such that
\[ \mathbb{P}^h=(\delta(p)\otimes\mathbb{P}^0)\otimes( \mathbb{P}\circ[\tau_h]^{-1}|_{\Omega_t}),
\]
where $\delta(p)$ is the measure under which $\bold p$ is constant and equal to $p$ and $\mathbb{P}^0$ is a Wiener measure on $\Omega_{0,t}$.
So by definition $(B_s)_{s\in[t,T]}$ is a Brownian motion under $\mathbb{P}^h$. Also $(\bold p_s)_{s\in[t,T]}$ is still a martingale under $\mathbb{P}^h$.
We can see this immediately, since for all $t\leq s\leq r\leq T$ by (26)
\[\mathbb{E}_{\mathbb{P}^h}[\bold p_r|\mathcal{F}_s]=\mathbb{E}_{\mathbb{P} \circ [\tau_h]^{-1}}[\bold p_r|\mathcal{F}_s]=\mathbb{E}_{\mathbb{P}}[\bold p_r|\mathcal{F}_s].\]
Furthermore the remaining conditions of Definition 3.1. are obviously met. Hence $\mathbb{P}^h\in\mathcal{P}(t,p)$ and, since $Y^{t,x,\mathbb{P}\circ [\tau_h]^{-1}}$ is a solution of a BSDE, we have
\[Y^{t,x,\mathbb{P}\circ[\tau_h]^{-1}}_{t-}=Y^{t,x,\mathbb{P}^h}_{t-}.
\]
On the other hand by considering $\mathbb{P}\circ\tau_h$  one can associate to any $\mathbb{P}\in\mathcal{P}(t,p)$ a $\mathbb{P}^{-h}\in\mathcal{P}(t,p)$, such that 
\[Y^{t,x,\mathbb{P}^{-h}\circ[\tau_h]^{-1}}_{t-}=Y^{t,x,\mathbb{P}}_{t-}.
\]
Hence $\left\{Y^{t,x,\mathbb{P}\circ[\tau_h]^{-1}}_{t-}: \mathbb{P}\in\mathcal{P}(t,p)\right\}=\left\{Y^{t,x,\mathbb{P}}_{t-}: \mathbb{P}\in\mathcal{P}(t,p)\right\}$ and
\begin{eqnarray*}
\textnormal{essinf}_{\mathbb{P}\in \mathcal{P}(t,p)} Y^{t,x,\mathbb{P}\circ[\tau_h]^{-1}}_{t-}
	= \textnormal{essinf}_{\mathbb{P}\in \mathcal{P}(t,p)} Y^{t,x,\mathbb{P}}_{t-}=W(t,x,p).
\end{eqnarray*}

\end{proof}

Proposition 3.6. follows then by Lemma 4.1. in \cite{BuLi}.\\

In the following section we establish some regularity results and a dynamic programming principle. To this end we work with $\epsilon$-optimal measures. Note that since we are taking the essential infimum over a family of random variables, existence of an $\epsilon$-optimal $\mathbb{P}^\epsilon\in \mathcal{P}(t,p)$ is not standard. Therefore we provide a technical lemma, the proof of which is also strongly inspired by \cite{BuLi}.

\begin{lem}
For any $(t,x,p)\in[0,T[\times\mathbb{R}^d\times\Delta(I)$ there is an $\epsilon$-optimal $\mathbb{P}^\epsilon\in \mathcal{P}(t,p)$ in the sense that
\begin{equation*}
Y^{t,x,\mathbb{P}^\epsilon}_{t-}\leq W(t,x,p)+\epsilon\ \ \ \ \ \ \ \ \ \mathbb{Q}\textnormal{-a.s.}
\end{equation*}
\end{lem}

\begin{proof}
Note that there exists a sequence $(\mathbb{P}^n)_{n\in\mathbb{N}}$, $\mathbb{P}^n\in \mathcal{P}(t,p)$, such that
\begin{eqnarray*}
W(t,x,p)=\textnormal{essinf}_{\mathbb{P} \in \mathcal{P}(t,p)}  Y^{t,x,\mathbb{P}}_{t-}=\inf_{n\in\mathbb{N}}Y^{t,x,\mathbb{P}^n}_{t-}.
\end{eqnarray*}
For an $\epsilon>0$ set $\Gamma_n:=\{W(t,x,p)+\epsilon\geq Y^{t,x,\mathbb{P}^n}_{t-} \}\in\mathcal{F}_{t-}$ for any $n\in\mathbb{N}$. Then $\bar{\Gamma}_1:=\Gamma_1$, $\bar{\Gamma}_n:=\Gamma_n\setminus (\cup_{m=1,\ldots,n-1}\bar{\Gamma}_m)$ for $n\geq2$ form a $\mathcal{F}_{t-}$ measurable partition of $\Omega$.\\
We define $\mathbb{P}^\epsilon$, such that on $\Omega=\Omega_{0,t}\times\Omega_t$
\[ \mathbb{P}^\epsilon=(\delta(p)\otimes\mathbb{P}^0)\otimes\left(\sum_{n\in\mathbb{N}} 1_{\bar\Gamma_n}\mathbb{P}^n|_{\Omega_t}\right),
\]
where $\delta(p)$ denotes the measure under which $\bold p$ is constant and equal to $p$ and $\mathbb{P}^0$ is a Wiener measure on $\Omega_{0,t}$.\\
So by definition $(B_s)_{s\in[t,T]}$ is a Brownian motion under $\mathbb{P}^\epsilon$ and $(\bold p_s)_{s\in[t,T]}$ is still a martingale under $\mathbb{P}^\epsilon$, since for all $t\leq s\leq r\leq T$ 
\[\mathbb{E}_{\mathbb{P}^\epsilon}[\bold p_r|\mathcal{F}_s]=\sum_{n\in\mathbb{N}}  \mathbb{E}_{\mathbb{P}^n}[1_{\bar{\Gamma}_n}\bold p_r|\mathcal{F}_s]=\sum_{n\in\mathbb{N}}  1_{\bar{\Gamma}_n} \mathbb{E}_{\mathbb{P}^n}[\bold p_r|\mathcal{F}_s]=\sum_{n\in\mathbb{N}}  1_{\bar{\Gamma}_n}\bold p_s=\bold p_s.\]
Again the remaining conditions of Definition 3.1. are obviously met. Thus $\mathbb{P}^\epsilon\in \mathcal{P}(t,p)$ and it holds 
\begin{eqnarray*}
W(t,x,p)+\epsilon\geq \sum_{n\in\mathbb{N}} 1_{\bar{\Gamma}_n} Y^{t,x,\mathbb{P}^n}_{t-} = Y^{t,x,\mathbb{P}^\epsilon}_{t-}.
\end{eqnarray*}
\end{proof}

\subsection{Some regularity results}

For technical reasons we will consider the BSDE (10) with a slightly different notation. For any $t\in[0,T]$, $x\in\mathbb{R}^d$, ${\mathbb{P} \in \mathcal{P}(t,p)}$ let
\begin{eqnarray}
Y_s^{t,x,\mathbb{P}}&=&\langle\bold p_T, g(X_T^{t,x})\rangle+\int_s^T \tilde H(r,X_r^{t,x},z^{t,x,\mathbb{P}}_r,\bold{p}_r)dr-\int_s^Tz^{t,x,\mathbb{P}}_rdB_r-N_T+N_s,
\end{eqnarray}
where $\tilde {H}(t,x,p,\xi) =H(t,x,p,(\sigma^*(t,x))^{-1}\xi)$. Setting $Z^{t,x,\mathbb{P}}_s= (\sigma^*(s,X_s^{t,x}))^{-1}z^{t,x,\mathbb{P}}_s$ then gives the solution to (10).\\
In the following we will use the notation $Y_s^{t,x,\mathbb{P}}=Y_s^{t,x}$, $z^{t,x,\mathbb{P}}=z^{t,x}$, whenever we work under a fixed ${\mathbb{P} \in \mathcal{P}(t,p)}$.

\begin{rem}
Observe that  by (H) we have that $\tilde {H}$ is uniformly Lipschitz continuous in $(\xi,p)$ uniformly in $(t,x)$ and Lipschitz continuous in $(t,x)$ with Lipschitz constant $c(1+|\xi|)$, i.e. it holds for all $t,t'\in[0,T]$, $x,x'\in\mathbb{R}^d$, $\xi,\xi'\in\mathbb{R}^d$, $p,p'\in\Delta(I)$
\begin{eqnarray}
&&|\tilde H(t,x,\xi,p)|\leq c (1+|\xi|)
\end{eqnarray}
and
\begin{eqnarray}
&&|\tilde H(t,x,\xi,p)-\tilde H(t',x',\xi',p')|\leq c (1+|\xi|)(|x-x'|+|t-t'|)+c|\xi-\xi'|+ c |p-p'|.
\end{eqnarray}
\end{rem}

\begin{prop}
$W(t,x,p)$ is uniformly Lipschitz continuous in $x$ and uniformly H\"older continuous in $t$.
\end{prop}
\begin{proof}
For the Lipschitz continuity in $x$, assume $W(t,x',p)-W(t,x,p)>0$ and let $\mathbb{P}^\epsilon\in \mathcal{P}(t,p)$ be $\epsilon$-optimal for $W(t,x,p)$ for a sufficiently small $\epsilon$. Then, since $W(t,x',p), W(t,x,p)$ are deterministic, we have by H\"older inequality and Proposition A.3.
\begin{eqnarray*}
0&\leq& W(t,x',p)-W(t,x,p)-\epsilon\\
\ \ \ &\leq& \mathbb{E}_{\mathbb{P}^\epsilon}\left[\textnormal{essinf}_{\mathbb{P} \in \mathcal{P}(t,p)} \mathbb{E}_\mathbb{P}\left[ \int_t^T\tilde H(s,X^{t,x'}_s,z^{t,x',\mathbb{P} }_s,\bold p_s)ds+\langle\bold p_T,g(X^{t,x'}_T)\rangle\big|\mathcal{F}_{t-}\right]\right]\\
\ \ \ &&\ \ \ \ \  \ \ \ \ \  \ \ \ \ \  \ \ \ \ \  \ \ \ \ \ -\mathbb{E}_{\mathbb{P}^\epsilon}\left[ \int_t^T\tilde H(s,X^{t,x}_s,z^{t,x}_s,{\bold p}_s)ds+\langle\bold p_T,g(X^{t,x}_T)\rangle\right]\\
\ \ \  &\leq&  \mathbb{E}_{\mathbb{P}^\epsilon}\left[ \int_t^T\tilde H(s,X^{t,x'}_s,z^{t,x'}_s,{\bold p}_s)-\tilde H(s,X^{t,x}_s,z^{t,x}_s,{\bold p}_s)ds+\langle\bold p_T,g(X^{t,x'}_T)-g(X^{t,x}_T)\rangle\right]\\
\ \ \  &\leq&  c\mathbb{E}_{\mathbb{P}^\epsilon}\left[ \int_t^T\left((1+|z^{t,x}_s|)|X^{t,x}_s-X^{t,x'}_s|+|z^{t,x}_s-z^{t,x'}_s|\right)ds+|X^{t,x}_T-X^{t,x'}_T|\right]\\
\ \ \  &\leq&  c  \left(\mathbb{E}_{\mathbb{P}^\epsilon}\left[ \int_t^T|z^{t,x}_s|^2ds\right]\right)^\frac{1}{2}\left(\mathbb{E}_{\mathbb{P}^\epsilon}\left[ \int_t^T|X^{t,x}_s-X^{t,x'}_s|^2ds\right]\right)^\frac{1}{2}\\
&&\ \ \ \ \ \ \ \ \ \ + c\mathbb{E}_{\mathbb{P}^\epsilon}\left[ \int_t^T\left(|X^{t,x}_s-X^{t,x'}_s|+|z^{t,x}_s-z^{t,x'}_s|\right)ds+|X^{t,x}_T-X^{t,x'}_T|\right]\\
&\leq&  c \left(\mathbb{E}_{\mathbb{P}^\epsilon}\left[ \int_t^T|X^{t,x}_s-X^{t,x'}_s|^2ds+|X^{t,x}_T-X^{t,x'}_T|^2\right]\right)^\frac{1}{2}\leq c|x-x'|,
\end{eqnarray*}
since for any $s\in[t,T]$ one has by Gronwall $\mathbb{E}_{\mathbb{P}^\epsilon}\left[|X^{t,x}_s-X^{t,x'}_s|^2\right]\leq c |x-x'|^2$.\\
For the H\"older continuity in time, let $t,t'\in[0,T]$ such that $t'\leq t$
and assume $W(t',x,p)>W(t,x,p)$. Let $\mathbb{P}^\epsilon\in \mathcal{P}(t,p)$ be $\epsilon$-optimal for $W(t,x,p)$ for a sufficiently small $\epsilon$. Note that since $t'\leq t$ it holds $\mathbb{P}^\epsilon\in \mathcal{P}(t',p)$. Then, since $W(t',x,p), W(t,x,p)$ are deterministic, we have by H\"older inequality and Proposition A.3.
\begin{eqnarray*}
0&\leq&W(t',x,p)-W(t,x,p)-\epsilon\\
\ \ \  &\leq&\mathbb{E}_{\mathbb{P}^\epsilon}\left[\textnormal{essinf}_{\mathbb{P} \in \mathcal{P}(t',p)}\mathbb{E}_\mathbb{P}\left[ \int_{t'}^T\tilde H(s,X^{t',x}_s,z^{t',x,\mathbb{P} }_s,\bold p_s)ds+\langle\bold p_T,g(X^{t',x}_T)\rangle\big|\mathcal{F}_{t'-}\right]\right]\\
\ \ \ &&\ \ \ \ \  \ \ \ \ \  \ \ \ \ \  \ \ \ \ \  \ \ \ \ \ -\mathbb{E}_{\mathbb{P}^\epsilon}\left[ \int_t^T\tilde H(s,X^{t,x}_s,z^{t,x}_s,{\bold p}_s)ds+\langle\bold p_T,g(X^{t,x}_T)\rangle\right]\\
\ \ \ &\leq&  \mathbb{E}_{\mathbb{P}^\epsilon}\left[ \int_{t'}^t\tilde H(s,X^{t',x}_s,z^{t',x}_s,\bold p_s)ds\right]\\
&&\ \ \ \ \ \ \ \ +\mathbb{E}_{\mathbb{P}^\epsilon}\left[ \int_t^T\tilde H(s,X^{t',x}_s,z^{t',x}_s,{\bold p}_s)-\tilde H(s,X^{t,x}_s,z^{t,x}_s,{\bold p}_s)ds+\langle\bold p_T,g(X^{t',x}_T)-g(X^{t,x}_T)\rangle\right]\\
\ \ \  &\leq&  c\ \mathbb{E}_{\mathbb{P}^\epsilon}\left[ \int_{t'}^t(1+|z^{t',x}_s|)ds\right]\\
&&\ \ \ \ \ \ \ \ +c\ \mathbb{E}_{\mathbb{P}^\epsilon}\left[ \int_{t}^T\left((1+|z^{t',x}_s|)|X^{t',x}_s-X^{t,x}_s|+|z^{t',x}_s-z^{t,x}_s|\right)ds+|X^{t',x}_T-X^{t,x}_T|\right]
\end{eqnarray*}
\begin{eqnarray*}
\ \ \  &\leq&  c |t'-t|^\frac{1}{2}+c  \left(\mathbb{E}_{\mathbb{P}^\epsilon}\left[ \int_t^T|z^{t',x}_s|^2ds\right]\right)^\frac{1}{2}\left(\mathbb{E}_{\mathbb{P}^\epsilon}\left[ \int_t^T|X^{t,x}_s-X^{t',x}_s|^2ds\right]\right)^\frac{1}{2}\\
&&\ \ \ \ \ \ \ \ \ \ + c\mathbb{E}_{\mathbb{P}^\epsilon}\left[ \int_t^T\left(|X^{t',x}_s-X^{t,x}_s|+|z^{t',x}_s-z^{t,x}_s|\right)ds+|X^{t,x}_T-X^{t',x}_T|\right]\\
\ \ \  &\leq&  c |t'-t|^\frac{1}{2}+c \left(\mathbb{E}_{\mathbb{P}^\epsilon}\left[ \int_t^T|X^{t',x}_s-X^{t,x}_s|^2ds+|X^{t',x}_T-X^{t,x}_T|^2\right]\right)^\frac{1}{2}\\
\ \ \  &\leq&  c |t'-t|^\frac{1}{2},
\end{eqnarray*}
because for any $s\in[t,T]$ it holds $\mathbb{E}_{\mathbb{P}^*}\left[|X^{t',x}_s-X^{t,x}_s|^2\right]\leq c |t'-t|$.\\
For the case $t'\leq t$, $W(t',x,p)<W(t,x,p)$ choose a $\mathbb{P}^\epsilon\in \mathcal{P}(t',p)$, which is $\epsilon$-optimal for $W(t',x,p)$ for a sufficiently small $\epsilon$.
We define then the probability measure $\bar {\mathbb{P}}^\epsilon$, such that on $\Omega=\Omega_{0,t}\times\Omega_t$
\[ \mathbb{P}^\epsilon=(\delta(p)\otimes\mathbb{P}^0)\otimes \mathbb{P}^\epsilon|_{\Omega_t},
\]
where $\delta(p)$ denotes the measure under which $\bold p$ is constant and equal to $p$ and $\mathbb{P}^0$ is a Wiener measure on $\Omega_{0,t}$.
So by definition $(B_s)_{s\in[t,T]}$ is a Brownian motion under $\bar {\mathbb{P}}^\epsilon$. Furthermore the remaining conditions of Definition 3.1. are met, hence $\bar{\mathbb{P}}^\epsilon\in\mathcal{P}(t,p)$ and the same argument as above applies in that case.
\end{proof}
\ \\

\begin{prop}
$W(t,x,p)$ is convex and uniformly Lipschitz continuous with respect to $p$.
\end{prop}

\begin{proof}
To show the convexity in $p$ let $p_1,p_2\in\Delta(I)$ and let $\mathbb{P}^{1}\in\mathcal{P}(t,p_1)$, $\mathbb{P}^{2}\in\mathcal{P}(t,p_2)$ be $\epsilon$-optimal for $W(t,x,p_1)$, $W(t,x,p_2)$ respectively. For $\lambda\in[0,1]$ define a martingale measure $\mathbb{P}^\lambda\in\mathcal{P}(t,p_\lambda)$, such that for all measurable $\phi:\mathcal{D}([0,T];\Delta(I))\times\mathcal{C}([0,T];\mathbb{R}^d)\rightarrow\mathbb{R}_+$
\begin{eqnarray*}
\mathbb{E}_{\mathbb{P}^\lambda}[\phi(\bold p,B)]=\lambda \mathbb{E}_{\mathbb{P}^1}[\phi(\bold p,B)]+(1-\lambda)\mathbb{E}_{\mathbb{P}^2}[\phi(\bold p,B)].
\end{eqnarray*}
Observe that we just take two copies $\Omega_1$, $\Omega_2$ of the same space with weights $\lambda,(1-\lambda)$. So for the respective solutions of the BSDE (27) it holds
\begin{eqnarray*}
Y^{t,x,\mathbb{P}^\lambda}=1_{\{\Omega_1\}}Y^{t,x,\mathbb{P}^1}+1_{\{\Omega_2\}}Y^{t,x,\mathbb{P}^2}.
\end{eqnarray*}
Hence
\begin{eqnarray*}
W(t,x,p_\lambda)&\leq&Y^{t,x,\mathbb{P}^\lambda}_{t-}= 1_{\{\Omega_1\}}Y^{t,x,\mathbb{P}^1}_{t-}+1_{\{\Omega_2\}}Y^{t,x,\mathbb{P}^2}_{t-}
\leq  1_{\{\Omega_1\}} W(t,x,p_1)+1_{\{\Omega_2\}} W(t,x,p_2)+2\epsilon
\end{eqnarray*}
and the convexity follows by taking expectation, since $\epsilon$ can be chosen arbitrarily small.\\
Next we prove the uniform Lipschitz continuity in $p$. Since we have convexity in $p$, it suffices to show the Lipschitz continuity with respect to $p$ on the extreme points $e_i$. Observe that $\mathcal{P}(t,e_i)$ consists in the single probability measure $\delta(e_i)\otimes\mathbb{P}^0$, where $\delta(e_i)$ is the measure under which $\bold p$ is constant and equal to $e_i$ and $\mathbb{P}^0$ is a Wiener measure.\\
Assume $W(t,x,p)-W(t,x,e_i)>0$ and let $\mathbb{P}^\epsilon\in \mathcal{P}(t,p)$ be $\epsilon$-optimal for $W(t,x,p)$ for a sufficiently small $\epsilon$. Then
\begin{eqnarray*}
&&0\leq W(t,x,p)-W(t,x,e_i)-\epsilon\\
&&\ \ \  \leq  \mathbb{E}_{\mathbb{P}^\epsilon}\left[ \int_t^T \tilde H(s,X^{t,x}_s,z^{t,x}_s,\bold p_s)-\tilde H(s,X^{x,t}_s,z^{t,x,e_i}_s,e_i)ds+\langle\bold p_T-e_i, g(X_T^{x,t})\rangle\big|\mathcal{F}_{t-}\right]\\
&&\ \ \  \leq   Y^{t,x}_{t-}-Y^{t,x,e_i}_{t-}.
\end{eqnarray*}
By the uniform Lipschitz continuity of $\tilde H$ in $\xi$ and $p$ it holds
\begin{eqnarray*}
Y^{t,x}_{t-}-Y^{t,x,e_i}_{t-}&\leq& \langle\bold p_T-e_i, g(X_T^{x,t})\rangle+c\int_t^T\left(|z^{t,x}_s-z^{t,x,e_i}_s|+|\bold{p}_s-e_i|\right)ds\\
&&\ \ \ \  -\int_t^T (z^{t,x}_s-z^{t,x,e_i}_s) dB_s -(N-N^{e_i})_T+(N-N^{e_i})_{t-}\\
&\leq& c\left ( \int_t^T (1-(\bold p_s)_i)ds+1-(\bold p_T)_i\right)+c\ \int_t^T|z^{t,x}_s-z^{t,x,e_i}_s|ds \\
&&\ \ \ \ -\int_t^T  (z_s^{t,x}-z^{t,x,e_i}_s) dB_s - (N-N^{e_i})_T+(N-N^{e_i})_{t-},
\end{eqnarray*}
where we used the estimate
\begin{eqnarray*}
\int_t^T|\bold p_s-e_i|ds+|\bold p_T-e_i|\leq  c\left ( \int_t^T (1-(\bold p_s)_i)ds+1-(\bold p_T)_i\right).
\end{eqnarray*}
We define $\hat Y$ as the unique solution to the BSDE
\begin{eqnarray*}
\hat Y_{s}=c\left ( \int_s^T (1-(\bold p_r)_i)dr+1-(\bold p_T)_i\right)+c\ \int_s^T|\hat z_r|dr -\int_s^T \hat z_r dB_r
-(\hat N_T-\hat N_s).
\end{eqnarray*}
Then by comparison (Theorem A.4.) we have
\begin{eqnarray*}
Y_{t-}^{t,x}-Y^{t,x,e_i}_{t-}\leq \hat Y_{t-}.
\end{eqnarray*}
We claim that $\hat Y_{s}=\left(1-(\bold p_{s})_i\right)\tilde Y_{s}$, where $\tilde Y_{s}$ is on $s\in[t,T]$ the solution to 
\begin{eqnarray}
\tilde Y_{s}=c+c\ (T-s)+ \int_s^T|\tilde z_r|dr -\int_s^T \tilde z_r dB_r.
\end{eqnarray}
This follows directly by applying the It\^o folmula
\begin{eqnarray*}
\left( 1-(\bold p_s)_i\right)\tilde Y_s&=&c\ \left( 1-(\bold p_T)_i\right)+c \int_s^T\left(1-(\bold p_r)_i\right)dr+\int_s^T \left|\left(1-(\bold p_r)_i\right)\tilde z_r\right|ds\\
&&\ \ \ \ \ \ 	-\int_s^T \left(1-(\bold p_r)_i\right)\tilde z_rdB_r+\int_{s}^T\tilde Y_r d(\bold p_r)_i
\end{eqnarray*}
and identifying $\hat z_s = \left(1-(\bold p_s)_i\right)\tilde z_s$ and $\tilde N_s=\int_0^s \tilde Y_r d(\bold p_r)_i$ which is by the definition of $\mathcal P(t,p)$ strongly orthogonal to $\mathcal{I}^2(\mathbb{P}^\epsilon)$. Furthermore  
\begin{eqnarray*}
1-(\bold p_{t-})_i=1- p_i\leq c \sum_{j} |(p)_j-\delta_{ij}|\leq c\sqrt{I}|p-e_i|,
\end{eqnarray*}
hence
\begin{eqnarray*}
	Y^{t,x}_{t-}-Y^{t,x,e_i}_{t-} \leq \hat Y_{t-}= (1-(\bold p_{t-})_i)\tilde Y_{t-}\leq c\sqrt{I}|p-e_i|\tilde Y_{t-}.
\end{eqnarray*}
It is well known (see e.g. \cite{ElK}) that, the solution $\tilde Y$ to (30) is continuous, bounded in $\mathcal{L}^1$ and $\tilde Y_{t}$ is deterministic. So $\tilde Y_{t-}=\tilde Y_{t}\leq c$ and we have
 \begin{eqnarray*}
	Y^{t,x}_{t-}-Y^{t,x,e_i}_{t-} \leq c\sqrt{I}|p-e_i|.
\end{eqnarray*}

\end{proof}

\subsection{Dynamic Programming Principle}

Next we show that a dynamic programming principle holds. To that end we introduce the set $\mathcal{P}^{f}(t,p)$ as the set of all measures $\mathbb{P} \in  \mathcal{P}(t,p)$, such that there exists a finite set $S\subset\Delta (I)$ with $\bold p_s\in S$ $\mathbb{P}$-a.s. for all $s\in[t,T]$. It is well known (see e.g. \cite{Pr} Theorem II.4.10) that $\mathcal{P}^{f}(t,p)$ is dense in $\mathcal{P}(t,p)$ with respect to the weak$^*$ topology.\\

\begin{thm}
For all  $(t,x,p)\in[0,T]\times\mathbb{R}^d\times\Delta(I)$, $t'\in[t,T]$
\begin{eqnarray}
W(t,x,p)=\textnormal{essinf}_{\mathbb{P}\in \mathcal{P}(t,p)} \mathbb{E}_\mathbb{P}\left[\int_{t}^{t'} \tilde H(s,X_s^{t,x},z^{t,x,\mathbb{P}}_s,\bold p_s)ds+W(t',X_{t'}^{t,x},\bold p_{t'-})\big|\mathcal{F}_{t-}\right].
\end{eqnarray}
\end{thm}

Since $\mathcal{P}^{f}(t,p)$ is a dense subset of $\mathcal{P}(t,p)$ with respect to the weak$^*$ topology, it suffices to show
\begin{eqnarray}
W(t,x,p)= \textnormal{essinf}_{\mathbb{P}\in \mathcal{P}^{f}(t,p)} \mathbb{E}_\mathbb{P}\left[\int_{t}^{t'}  \tilde H(s,X_s^{t,x},z^{t,x,\mathbb{P}}_s,\bold p_s)ds+W(t',X_{t'}^{t,x},\bold p_{t'-})\big|\mathcal{F}_{t-}\right]
\end{eqnarray}
for all $(t,x,p)\in[0,T]\times\mathbb{R}^d\times\Delta(I)$.\\

For the proof of Theorem 4.7. we first show two Lemmas.

\begin{lem}
Under any $\mathbb{P}\in\mathcal{P}^{f}(t,p)$
\begin{eqnarray}
Y^{t,x,\mathbb{P}}_{t'-} \geq W(t',X^{t,x}_{t'},\bold p_{t'-}).
\end{eqnarray}
\end{lem}

\begin{proof}
Fix $\mathbb{P}\in\mathcal{P}^{f}(t,p)$ and $t'\in[t,T]$. Let $(A_l)_{l\in\mathbb{N}}$ be a partition of $\mathbb{R}^d$ in Borel sets, such that $\textnormal{diam}(A_l)\leq \epsilon$ and choose for any $l\in\mathbb{N}$ some $y^l\in A_l$.
Let $z^{t',y^l}$ denote the $z$ term of the solution of BSDE (27) with forward dynamics $X^{t',y^l}$ instead of $X^{t,x}$.
First observe that 
\begin{eqnarray*}
Y^{t,x}_{t'-}&=&\mathbb {E}_\mathbb{P}\left[\int_{t'}^{T}  \tilde H(s,X_s^{t,x},z^{t,x}_s,\bold p_s)ds+\langle\bold p_T,g(X^{t,x}_T) \rangle\big|\mathcal{F}_{t'-}\right]\\
&=&\sum_{l=1}^\infty \mathbb{E}_{\mathbb{P}}\left[\left(\int_{t'}^{T}  \tilde H(s,X_s^{t,x},z_s^{t,x},\bold p_s)ds+\langle\bold p_T,g(X^{t,x}_T) \rangle \right)\big|\mathcal{F}_{t'-}\right]1_{\{X^{t,x}_{t'}\in A^l\}}\\
&\geq& \sum_{l=1}^\infty \mathbb{E}_{\mathbb{P}}\left[\int_{t'}^{T}  \tilde H(s,X_s^{t',y^l},z^{t',y^l}_s,\bold p_s)ds+\langle\bold p_T,g(X^{t',y^l}_T) \rangle\big|\mathcal{F}_{t'-}\right]1_{\{X^{t,x}_{t'}\in A^l\}}\\
&&\ \ \ \ \ \ \ - c \sum_{l=1}^\infty \mathbb{E}_{\mathbb{P}}\left[\int_{t'}^{T}\left(|z^{t',y^l}_s-z^{x,t}_s|+(1+ |z^{x,t}_s|)|X^{t',y^l}_{s}-X^{t,x}_{s}|\right)  ds+|X^{t',y^l}_T-X^{t,x}_T|\big|\mathcal{F}_{t'-}\right]1_{\{X^{t,x}_{t'}\in A^l\}}
\end{eqnarray*}
where by H\"older inequality, Proposition A.3. and Gronwall inequality
\begin{eqnarray*}
&& \sum_{l=1}^\infty \mathbb{E}_{\mathbb{P}}\left[\int_{t'}^{T}\left(|z^{t',y^l}_s-z^{x,t}_s|+(1+ |z^{x,t}_s|)|X^{t',y^l}_{s}-X^{t,x}_{s}|\right)  ds+|X^{t',y^l}_T-X^{t,x}_T|\big|\mathcal{F}_{t'-}\right]1_{\{X^{t,x}_{t'}\in A^l\}}\\
&\leq&c \sum_{l=1}^\infty \mathbb{E}_{\mathbb{P}}\left[\int_{t'}^{T}\left(|z^{t',y^l}_s-z^{t,x}_s|^2+|X^{t',y^l}_{s}-X^{t,x}_{s}|^2\right)  ds+|X^{t',y^l}_T-X^{t,x}_T|^2\big|\mathcal{F}_{t'-}\right]^\frac{1}{2}1_{\{X^{t,x}_{t'}\in A^l\}}\\
&\leq& c \sum_{l=1}^\infty \mathbb{E}_{\mathbb{P}}\left[\int_{t'}^{T}|X^{t',y^l}_{s}-X^{t,x}_{s}|^2 ds+|X^{t',y^l}_{T}-X^{t,x}_{T}|^2 \big|\mathcal{F}_{t'-}\right]^\frac{1}{2}1_{\{X^{t,x}_{t'}\in A^l\}}\\
&\leq& c \sum_{l=1}^\infty \mathbb{E}_{\mathbb{P}}\left[|X^{t,x}_{t'}-y^l|^2\big|\mathcal{F}_{t'-}\right]^\frac{1}{2}1_{\{X^{t,x}_{t'}\in A^l\}}
\leq c \sum_{l=1}^\infty |X^{t,x}_{t'}-y^l|1_{\{X^{t,x}_{t'}\in A^l\}}
\leq  c\epsilon.
\end{eqnarray*}
Hence
\begin{eqnarray}
\sum_{l=1}^\infty Y^{t',y^l}_{t'-} 1_{\{X^{t,x}_{t'}\in A^l\}}- c \epsilon \leq Y^{t,x}_{t'-}\leq \sum_{l=1}^\infty Y^{t',y^l}_{t'-} 1_{\{X^{t,x}_{t'}\in A^l\}}+c \epsilon.
\end{eqnarray}
where the upper bound is given by similar argumentation.
Furthermore by assumption there exist $S=\{p^1,\ldots,p^k\}$, such that $\mathbb{P}[\bold p_{t'-}\in S]=1$. 
We define for $m=1,\ldots,k$ the probablility measures $\mathbb{P}^m$, such that on $\Omega=\Omega_{0,t}\times\Omega_t$
\[ \mathbb{P}^m=(\delta(p^m)\otimes\mathbb{P}^0)\otimes\left(1_{\{\bold p_{t'-}=p_m\}}\mathbb{P}|_{\Omega_t}\right),
\]
where $\delta(p^m)$ denotes the measure under which $\bold p$ is constant and equal to $p^m$ and $\mathbb{P}^0$ is a Wiener measure on $\Omega_{0,t}$.
So by definition $(B_s)_{s\in[t,T]}$ is a Brownian motion under $\mathbb{P}^m$ and $(\bold p_s)_{s\in[t,T]}$ is a martingale.
We see this, since for $t' \leq s\leq T$
\[
\mathbb{E}_{\mathbb{P}^m}[\bold p_s|\mathcal{F}_{t'-}]=\mathbb{E}_{\mathbb{P}}[1_{\{\bold p_{t'-}=p_m\}}\bold p_s|\mathcal{F}_{t'-}]=1_{\{\bold p_{t'-}=p^m\}}\bold p_{t'-}=p^m.
\]
Furthermore the remaining conditions of Definition 3.1. are met, hence  $\mathbb{P}^m\in\mathcal{P}^f(t,p)$ for $m=1,\ldots,k$ and
\begin{eqnarray*}
Y^{t',y^l,\mathbb{P}^m}_{t'-}1_{\{\bold p_{t'-}=p_m\}}\geq W(t',y^l,p^m)1_{\bold p_{t'-}=p^m}.
\end{eqnarray*}
So it holds
\begin{eqnarray*}
Y^{t',y^l,\mathbb{P}}_{t'-} =\sum_{m=1}^{k} Y^{t',y^l,\mathbb{P}^m}_{t'-}1_{\{\bold p_{t'-}=p_m\}}\geq \sum_{m=1}^{k}W(t',y^l,p^m)1_{\{\bold p_{t'-}=p_m\}}= W(t',y^l,\bold p_{t'-}).
\end{eqnarray*}
Since $W$ is uniformly Lipschitz continuous in $x$, we have with (34)
\begin{eqnarray*}
Y^{t,x,\mathbb{P}}_{t'-} \geq W(t',X^{x,t}_{t'},\bold p_{t'-})-c\epsilon.
\end{eqnarray*}
for an arbitrarily small $\epsilon>0$.
\end{proof}
\ \\

\begin{lem}
For any $\epsilon>0$, $ t'\in[t,T]$ and $\mathbb{P}\in\mathcal{P}^{f}(t,p)$ one can choose a $\mathbb{P}^\epsilon\in\mathcal{P}^{f}(t,p)$, such that 
\begin{itemize}
\item [(i)]	${\mathbb{P}^\epsilon}=\mathbb{P}$ on $\mathcal{F}_{t'-}$ 
\item [(ii)]	and it holds
			\begin{eqnarray}
				Y^{t,x,{\mathbb{P}^\epsilon}}_{t'-}\leq W(t',X_{t'}^{t,x},\bold p_{t'-})+\epsilon.
			\end{eqnarray}
\end{itemize}
\end{lem}
\begin{rem}
Observe that by (i) it holds
\begin{eqnarray*}
\mathbb{E}_{\mathbb{P}}\left[\int_{t}^{t'}  \tilde H(s,X_s^{t,x},z^{t,x,{\mathbb{P}}}_s,\bold p_s)ds\big|\mathcal{F}_{t-}\right]&{=}&\mathbb{E}_{{\mathbb{P}}^\epsilon}\left[\int_{t}^{t'}  \tilde H(s,X_s^{t,x},z^{t,x,{\mathbb{P}}}_s,\bold p_s)ds\big|\mathcal{F}_{t-}\right],
\end{eqnarray*}
while by (ii) and Lemma 4.8.
\begin{eqnarray*}
Y^{t,x,{\mathbb{P}^\epsilon}}_{t'-}\leq W(t',X_{t'}^{t,x},\bold p_{t'-})+ \epsilon \leq Y^{t,x,{\mathbb{P}}}_{t'-}+\epsilon,
\end{eqnarray*}
hence by comparison (Theorem A.4.)
\begin{eqnarray}
\mathbb{E}_{{\mathbb{P}}^\epsilon}\left[\int_{t}^{t'}  \tilde H(s,X_s^{t,x},z^{t,x,\mathbb{P}^\epsilon}_s,\bold p_s)-\tilde H(s,X_s^{t,x},z^{t,x,{\mathbb{P}}}_s,\bold p_s)ds\big|\mathcal{F}_{t-}\right]&{\leq}&\epsilon.
\end{eqnarray}
\end{rem}

\begin{proof} (Lemma 4.9.)
Fix a $\mathbb{P}\in\mathcal{P}^f(t,p)$. Let $t'\in[t,T]$. By assumption there exist $S=\{p^1,\ldots,p^k\}$, such that $\mathbb{P}[\bold p_{t'-}\in S]=1$. Furthermore let $(A_l)_{l\in\mathbb{N}}$ be a partition of $\mathbb{R}^d$ by Borel sets, such that diam$(A_l)\leq \bar \epsilon$ and choose for any $l\in\mathbb{N}$ some $y^l\in A_l$.\\
Define for any $l,m$ measures $\mathbb{P}^{l,m}\in\mathcal{P}^f(t',p^m)$, such that
\begin{eqnarray*}
&&\mathbb{E}_{\mathbb{P}^{l,m}}\left[\int_{t'}^T  \tilde H(s,X_s^{t',y^l},z^{t',y^l,{\mathbb{P}^{l,m}}}_s,\bold p_s) ds+\langle\bold p_T, g(X_T^{t',y^l})\rangle\big|\mathcal{F}_{t'-}\right]\\
&&\ \ \ \ \leq \inf_{\mathbb{P}\in \mathcal{P}^f(t',p^m)}  \mathbb{E}_{\mathbb{P}}\left[\int_{t'}^T  \tilde H(s,X_s^{t',y^l},z^{t',y^l,\mathbb{P}}_s,\bold p_s) ds+\langle\bold p_T, g(X_T^{t',y^l})\rangle\big|\mathcal{F}_{t'-}\right]+\epsilon\\
&&\ \ \ \ =W(t',y^l,p^m)+\epsilon.
\end{eqnarray*}
We define the probablility measures ${\mathbb{P}}^\epsilon$, such that on $\Omega=\Omega_{0,t'}\times\Omega_{t'}$
\[ \mathbb{P}^\epsilon=(\mathbb{P}|_{\Omega_{0,t'}})\otimes\left(\sum_{m=1}^k\sum_{l=1}^\infty 1_{\{X_{t'}^{t,x}\in A^l,\bold p_{t'-}=p_m\}}\mathbb{P}^{l,m}|_{\Omega_{t'}}\right).
\]
So by definition $(B_s)_{s\in[t,T]}$ is a Brownian motion under $\mathbb{P}^\epsilon$. Also $(\bold p_s)_{s\in[t,T]}$ is a martingale, since for $ t'\leq r\leq s\leq T$
\[
\mathbb{E}_{\mathbb{P}^\epsilon}[\bold p_s|\mathcal{F}_{r}]=\sum_{m=1}^k \sum_{l=1}^\infty 1_{\{X_{t'}^{t,x}\in A^l,\bold p_{t'-}=p_m\}}\mathbb{E}_{\mathbb{P}^{l,m}}[\bold p_s|\mathcal{F}_{r}]=\sum_{m=1}^k \sum_{l=1}^\infty 1_{\{X_{t'}^{t,x}\in A^l,\bold p_{t'-}=p_m\}}\bold p_r=\bold p_r.
\]
Furthermore the remaining conditions of Definition 3.1. are obviously met, hence $\mathbb{P}^\epsilon\in\mathcal{P}^{f}(t,p)$.\\
Note that by the uniform Lipschitz continuity of $\tilde H$ and Proposition A.3.  we have as in (34)
\begin{eqnarray*}
Y^{t,x,{\mathbb{P}^\epsilon}}_{t'-}&=&\sum_{m=1}^k\sum_{l=1}^\infty 1_{\{X_{t'}^{t,x}\in A^l,\bold p_{t'-}=p_m\}} Y^{t,x,{\mathbb{P}^\epsilon}}_{t'-}\\
&\leq&  \sum_{m=1}^k\sum_{l=1}^\infty 1_{\{X_{t'}^{t,x}\in A^l,\bold p_{t'-}=p_m\}}  Y^{t',y^l,{\mathbb{P}^\epsilon}}_{t'-}+ c \sum_{m=1}^k\sum_{l=1}^\infty \mathbb{E}_{\mathbb{P}}\left[  1_{\{X_{t'}^{t,x}\in A^l,\bold p_{t'-}=p_m\}}|{X}^{t,x}_{t'}-y^l|\big|\mathcal{F}_{t'-}\right]\\
&\leq&  \sum_{m=1}^k\sum_{l=1}^\infty 1_{\{X_{t'}^{t,x}\in A^l,\bold p_{t'-}=p_m\}}  Y^{t',y^l,{\mathbb{P}^\epsilon}}_{t'-}+ c \bar \epsilon.
\end{eqnarray*}
So it holds by the definition of $\mathbb{P}^\epsilon$
\begin{eqnarray*}
Y^{t,x,{\mathbb{P}^\epsilon}}_{t'-}
&\leq&   \sum_{m=1}^k\sum_{l=1}^\infty 1_{\{X_{t'}^{t,x}\in A^l,\bold p_{t'-}=p_m\}}  Y^{t',y^l,{\mathbb{P}^{l,m}}}_{t'-} + c \bar \epsilon \\
&\leq&  \sum_{m=1}^k\sum_{l=1}^\infty 1_{\{X_{t'}^{t,x}\in A^l,\bold p_{t'-}=p_m\}} W(t',y^l,p^m)+\epsilon+ c \bar \epsilon\\
&\leq&W(t',X_{t'}^{t,x},\bold p_{t'-})+\epsilon+c\bar \epsilon
\end{eqnarray*}
and the result follows, since $\bar \epsilon$ can be chosen arbitrarily small.
\end{proof}

We are now ready to prove Theorem 4.7.

\begin{proof} (Theorem 4.7.) 
Let ${\mathbb{P}^\epsilon}\in\mathcal{P}^{f}(t,p)$ be ${\epsilon}$-optimal for $W(t,x,p)$. Then by Lemma 4.8.
\begin{eqnarray*}
W(t,x,p)+\epsilon&\geq& \mathbb{E}_{{\mathbb{P}^\epsilon}}\left[\int_{t}^{T}  \tilde H(s,X_s^{t,x},z_s^{t,x},\bold p_s)ds+\langle\bold p_T,g(X^{t,x}_T) \rangle\big|\mathcal{F}_{t-}\right]\\
&=& \mathbb{E}_{{\mathbb{P}^\epsilon}} \left[\int_{t}^{t'}  \tilde H(s,X_s^{t,x},z_s^{t,x},\bold p_s)ds + \int_{t'}^{T}  \tilde H(s,X_s^{t,x},z^{t,x}_s,\bold p_s)ds+\langle\bold p_T,g(X^{t,x}_T) \rangle\big|\mathcal{F}_{t-}\right]\\
&=& \mathbb{E}_{{\mathbb{P}^\epsilon}} \left[\int_{t}^{t'}  \tilde H(s,X_s^{t,x},z_s^{t,x},\bold p_s)ds + Y^{t,x}_{t'-}\big|\mathcal{F}_{t-}\right]\\
&\geq& \mathbb{E}_{{\mathbb{P}^\epsilon}} \left[\int_{t}^{t'}  \tilde H(s,X_s^{t,x},z_s^{t,x},\bold p_s)ds+W(t',X^{t,x}_{t'},\bold p_{t'-})\big|\mathcal{F}_{t-}\right].
\end{eqnarray*}
To prove the reverse inequality choose $\mathbb{P}^{\epsilon_1}\in\mathcal{P}^{f}(t,p)$ to be $\epsilon_1$ optimal for the RHS of (32), i.e.
\begin{equation}
\begin{array}{rcl}
&&\textnormal{essinf}_{\mathbb{P}\in \mathcal{P}^f(t,p)} \mathbb{E}_{\mathbb{P}}\left[\int_{t}^{t'}  \tilde H(s,X_s^{t,x},z^{t,x,\mathbb{P}}_s,\bold p_s)ds+W(t',X_{t'}^{t,x},\bold p_{t'-})\big|\mathcal{F}_{t-}\right]+\epsilon_1\\
\ \\
&&\ \ \ \ \geq \mathbb{E}_{\mathbb{P}^{\epsilon_1}}\left[\int_{t}^{t'}  \tilde H(s,X_s^{t,x},z^{t,x,\mathbb{P}^{\epsilon_1}}_s,\bold p_s)ds+W(t',X_{t'}^{t,x},\bold p_{t'-})\big|\mathcal{F}_{t-}\right]
\end{array}
\end{equation}
Furthermore choose as in Lemma 4.9. for $\mathbb{P}^{\epsilon_1}$ a $\mathbb{P}^{\epsilon_{1,2}}\in\mathcal{P}^{f}(t,p)$ to be $\epsilon_ 2$ optimal. Then by (35), (36)
\begin{equation}
\begin{array}{rcl}
&&\mathbb{E}_{\mathbb{P}^{\epsilon_1}}\left[\int_{t}^{t'}  \tilde H(s,X_s^{t,x},z^{t,x,\mathbb{P}^{\epsilon_1}}_s,\bold p_s)ds+W(t',X_{t'}^{t,x},\bold p_{t'-})\big|\mathcal{F}_{t-}\right]+2\epsilon_2\\
\ \\
&&\ \ \ \geq \mathbb{E}_{\mathbb{P}^{\epsilon_{1,2}}}\left[\int_{t}^{t'}  \tilde H(s,X_s^{t,x},z^{t,x,\mathbb{P}^{\epsilon_{1,2}}}_s,\bold p_s)ds+Y^{t,x,\mathbb{P}^{\epsilon_{1,2}}}_{t'-}\big|\mathcal{F}_{t-}\right]=Y^{t,x,\mathbb{P}^{\epsilon_{1,2}}}_{t-}
\end{array}
\end{equation}
Finally combining (37), (38) we have
\begin{eqnarray*}
&&\textnormal{essinf}_{\mathbb{P}\in \mathcal{P}^f(t,p)} \mathbb{E}_{\mathbb{P}}\left[\int_{t}^{t'}  \tilde H(s,X_s^{t,x},z^{t,x,\mathbb{P}}_s,\bold p_s)ds+W(t',X_{t'}^{t,x},\bold p_{t'})\big|\mathcal{F}_{t-}\right]+\epsilon_1+2\epsilon_2\\
&&\ \ \ \geq Y^{t,x,\mathbb{P}^{\epsilon_{1,2}}}_{t-}\geq W(t,x,p).
\end{eqnarray*}

\end{proof}

\subsection{Viscosity solution property}

To proof that $W$ is a viscosity solution to (8) we first show the subsolution property which is an easy consequence of the Dynamic Programming Theorem 4.7.

\begin{prop}
$W$ is a viscosity subsolution to (8) on $[0,T]\times\mathbb{R}^d\times\textnormal{Int}(\Delta(I)).$
\end{prop}

\begin{proof}
Let $\phi:[0,T]\times\mathbb{R}^d\times\Delta(I)\rightarrow \mathbb{R}$ be a test function such that $W-\phi$ has a strict global maximum at $(\bar{t},\bar x,\bar p)$ with $W(\bar{t},\bar x,\bar p)-\phi(\bar{t},\bar x,\bar p)=0$ and $\bar p\in\textnormal{Int}(\Delta (I))$.
We have to show, that
\begin{eqnarray}
	&&\min\bigg\{\frac{\partial \phi}{\partial t}+\frac{1}{2}\textnormal{tr}(\sigma\sigma^*(t,x)D_x^2\phi)+H(t,x,D_x\phi,p), \lambda_{\min} \left(\frac{\partial^2 \phi}{\partial p^2}\right)\bigg\}\geq 0
\end{eqnarray}
holds at $(\bar{t},\bar x,\bar p)$.\\
By Proposition 4.6. $W$ is convex in $p$. So since $\bar p\in\textnormal{Int}(\Delta (I))$, it holds $\lambda_{\min}\left(\frac{\partial ^2 \phi}{\partial p^2}(\bar{t},\bar x,\bar p)\right)\geq 0$. Furthermore
\begin{eqnarray*}
&&\phi(\bar{t},\bar x,\bar p)=W(\bar t,\bar x,\bar p)=\textnormal{essinf}_{\mathbb{P}\in \mathcal{P}^r(\bar t,\bar p)} \mathbb{E}\left[\int_{\bar t}^{t} \tilde H(s,X_s^{\bar{t},\bar x},z^{\bar{t},\bar x}_s,\bold p_s)ds+W(t,X_{t}^{\bar{t},\bar x},\bold p_{t-})\big|\mathcal{F}_{\bar t-}\right]\\
&&\ \ \ \ \ \ \ \ \ \ \ \  \ \leq \mathbb{E}\left[\int_{\bar t}^{t} \tilde H(s,X_s^{\bar{t},\bar x},z_s^{\bar{t},\bar x},\bar p)ds+W(t,X_{t}^{\bar{t},\bar x},\bar p)\big|\mathcal{F}_{\bar t-}\right].
\end{eqnarray*}
Since by standard Markov arguments  $\mathbb{E}\left[\int_{\bar t}^{t} \tilde H(s,X_s^{\bar{t},\bar x},z_s^{\bar{t},\bar x},\bar p)ds+W(t,X_{t}^{\bar{t},\bar x},\bar p)\big|\mathcal{F}_{\bar t-}\right]$ is deterministic and $W\leq\phi$ by construction, this yields
\begin{eqnarray*}
\phi(\bar{t},\bar x,\bar p)\leq \mathbb{E}\left[\int_{\bar t}^{t} \tilde H(s,X_s^{\bar{t},\bar x},z^{\bar{t},\bar x}_s,\bar p)ds+\phi(t,X_{t}^{\bar{t},\bar x},\bar p)\right],
\end{eqnarray*}
which implies (39) as $t\downarrow\bar t$ by standard results (see e.g. \cite{ElK}). 
\end{proof}

\begin{prop}
$W$ is a viscosity supersolution to (8) on $[0,T]\times\mathbb{R}^d\times\Delta(I)$.
\end{prop}

\begin{proof}
Let $\phi:[0,T]\times\mathbb{R}^d\times\Delta(I)\rightarrow \mathbb{R}$ be a smooth test function, such that $W-\phi$ has a strict global minimum at $(\bar{t},\bar x,\bar p)$ with $W(\bar{t},\bar x,\bar p)-\phi(\bar{t},\bar x,\bar p)=0$ and such that its derivatives are uniformly Lipschitz in $p$.\\
We have to show, that
\begin{eqnarray}
	&&\min\bigg\{\frac{\partial \phi}{\partial t}+\frac{1}{2}\textnormal{tr}(\sigma\sigma^*(t,x)D_x^2\phi)+H(t,x,D\phi,p), \lambda_{\min} \left(\frac{\partial^2 \phi}{\partial p^2}\right)\bigg\}\leq 0
\end{eqnarray}
holds at $(\bar{t},\bar x,\bar p)$. Observe that, if $\lambda_{\min} \left(\frac{\partial^2 \phi}{\partial p^2}\right)\leq0$ at $(\bar{t},\bar x,\bar p)$, then (40) follows immediately.\\
We assume in the subsequent steps strict convexity of $\phi$ in $p$ at $(\bar{t},\bar x,\bar p)$, i.e. there exist $\delta,\eta>0$ such that for all $z\in T_{\Delta(I)(\bar p)}$
\begin{eqnarray}
\langle \frac{\partial^2 \phi}{\partial p^2}(t,x,p)z,z\rangle>4\delta|z|^2\ \ \ \ \ \ \forall (t,x,p)\in B_\eta(\bar t,\bar x,\bar p).
\end{eqnarray}
Since $\phi$ is a test function for a purely local viscosity notion, one can modify it outside a neighborhood of $(\bar t,\bar x,\bar p)$ such that for all $(s, x)\in [\bar t,T]\times\mathbb{R}^d$ the function $\phi(s,x,\cdot)$ is convex on the whole convex domain $\Delta(I)$. Thus for any $p\in\Delta(I)$ it holds
\begin{eqnarray}
W(t,x,p)\geq \phi(t,x,p)\geq \phi(t,x,\bar p)+\langle\frac{\partial \phi }{\partial p}(t,x,p),p-\bar p\rangle.
\end{eqnarray}
We divide the proof in several steps. First we show an estimate which is stronger than (42) basing on the strict convexity assumption (41). In the second step we use the dynamic programming to establish estimates for $\bold p$. The subsequent steps are rather close to the standard case. We reduce the problem by considering a BSDE on a smaller time interval. Then we establish estimates for the auxiliary BSDE, which we use in the last step to show the viscosity supersolution property.\\

\textbf{Step 1}: 
We claim that there exist $\eta,\delta>0$, such that for all $(t,x)\in B_\eta(\bar t,\bar x)$, $p\in\Delta(I)$
\begin{eqnarray}
W(t,x,p)\geq \phi(t,x,\bar p)+\langle\frac{\partial \phi }{\partial p}(t,x,p),p-\bar p\rangle+2\delta|p-\bar p|^2.
\end{eqnarray}
By Taylor expansion in $p$ we have for all $(t,x,p)\in B_\eta(\bar t,\bar x,\bar p)$
\begin{eqnarray}
W(t,x,p)\geq \phi(t,x,p)\geq \phi(t,x,\bar p)+\langle\frac{\partial \phi }{\partial p}(t,x,p),p-\bar p\rangle+2\delta|p-\bar p|^2.
\end{eqnarray}
To establish (44) for all $p\in\Delta(I)$ we set for $p\in\Delta(I)\setminus \textnormal{Int}(B_\eta(\bar p))$
\[ \tilde p = \bar p+\frac{p-\bar p}{|p-\bar p|}\eta.\]
By the convexity of $W$ in $p$ and (44) we have for any $\hat{ p}\in\partial {W}_p^-(\bar t,\bar x, \tilde p)$
\begin{eqnarray*}
W(\bar t,\bar x,p) &\geq& W(\bar t,\bar x,\tilde p)+\langle \hat p,p-\tilde p \rangle\\
&\geq& \phi(\bar t,\bar x,\bar p)+\langle\frac{\partial \phi}{\partial p}(\bar t,\bar x,\bar p),\tilde p-\bar p\rangle+2\delta\eta^2+\langle \hat p,p-\tilde p \rangle\\
&\geq& \phi(\bar t,\bar x,\bar p)+\langle\frac{\partial \phi}{\partial p}(\bar t,\bar x,\bar p), p-\bar p\rangle+2\delta\eta^2+\langle \hat p-\frac{\partial \phi }{\partial p}(\bar t,\bar x,\bar p),p-\tilde p \rangle.
\end{eqnarray*}
Since $\frac{\partial \phi}{\partial p}(\bar t,\bar x,\bar p)\in\partial{W}_p^-(\bar t,\bar x, \bar p)$ and $p-\tilde p=c (p-\bar p)$ $(c>0)$ and $W$ is convex in $p$, it holds
\[\langle \hat p-\frac{\partial \phi}{\partial p}(\bar t,\bar x,\bar p),p-\tilde p \rangle\geq0.\]
So we have for all $p\in\Delta(I)\setminus \textnormal{Int}(B_\eta(\bar p))$
\begin{eqnarray}
W(\bar t,\bar x,p) \geq \phi(\bar t,\bar x,\bar p)+\langle\frac{\partial \phi}{\partial p}(\bar t,\bar x,\bar p), p-\bar p\rangle+2\delta\eta^2.
\end{eqnarray}
Assume now that (43) does not hold for a $p\in\Delta(I)$. Then there exists a sequence $(t_k,x_{k},p_{k})\rightarrow(\bar t, \bar x, p)$ with ${p_{k}}\in\Delta(I)\setminus B_\eta(\bar p)$, such that
\begin{eqnarray*}
 W(t_k,x_{k},p_{k})< \phi(t_{k},x_{k},{p}_{k})+\langle\frac{\partial \phi}{\partial p}(t_{k},x_{k},p_{k}),p_{k}-\bar p\rangle+\delta|p_{k}-\bar p|^2.
\end{eqnarray*}
Thus for $k\rightarrow\infty$, $p\in\Delta(I)\setminus \textnormal{Int}(B_\eta(\bar p))$ and 
\begin{eqnarray*}
W(\bar t,\bar x,p)< \phi(\bar t,\bar x,\bar p)+\langle\frac{\partial \phi }{\partial p}(\bar t,\bar x,\bar p), p-\bar p\rangle+\delta\eta^2
\end{eqnarray*}
which contradicts (45).\\
Note that by (43) we have for any  $t>\bar t$ such that $(t-\bar t)$ is sufficiently small and an $\eta'<\eta$
\begin{eqnarray*}
W(t,X^{\bar t,\bar x}_{t},\bold p_{t-})&=&1_{\{|X^{\bar t,\bar x}_{t}-\bar x|<\eta'\}} W(t,X^{\bar t,\bar x}_{t},\bold p_{t-})+1_{\{|X^{\bar t,\bar x}_{t}-\bar x|\geq\eta'\}}  W(t,X^{\bar t,\bar x}_{t},\bold p_{t-})\\
&\geq&1_{\{|X^{\bar t,\bar x}_{t}-\bar x|<\eta'\}}  \left(\phi(t,X^{\bar t,\bar x}_{t},\bar p)+\langle\frac{\partial \phi }{\partial p}(t,X^{\bar t,\bar x}_{t},\bar p),\bold p_{t-}-\bar p\rangle+\delta|\bold p_{t-}-\bar p|^2\right)\\
&&\ \ \  +1_{\{|X^{\bar t,\bar x}_{t}-\bar x|\geq\eta'\}} \phi(t,X^{\bar t,\bar x}_{t},\bold p_{t-})\\
&\geq& \phi(t,X^{\bar t,\bar x}_{t},\bar p)+\langle\frac{\partial \phi }{\partial p}(t,X^{\bar t,\bar x}_{t},\bar p),\bold p_{t-}-\bar p\rangle+1_{\{|X^{\bar t,\bar x}_{t}-\bar x|<\eta'\}} \delta|\bold p_{t-}-\bar p|^2\\
&&\ \ \  +1_{\{|X^{\bar t,\bar x}_{t}-\bar x|\geq\eta'\}}  \left(\phi(t,X^{\bar t,\bar x}_{t}, \bold p_{t-})-\phi(t,X^{\bar t,\bar x}_{t},\bar p)-\langle\frac{\partial \phi }{\partial p}(t,X^{\bar t,\bar x}_{t},\bar p),\bold p_{t-}-\bar p\rangle\right)
\end{eqnarray*}
Recalling that $\phi$ is convex with respect to $p$, we get
\begin{equation}
\begin{array}{rcl}
W(t,X^{\bar t,\bar x}_{t},\bold p_{t-})\geq \phi(t,X^{\bar t,\bar x}_{t},\bar p)+\langle\frac{\partial \phi }{\partial p}(t,X^{\bar t,\bar x}_{t},\bar p),\bold p_{t-}-\bar p\rangle
+\delta 1_{\{|X^{\bar t,\bar x}_{t}-\bar x|<\eta'\}}|\bold p_{t-}-\bar p|^2.
\end{array}
\end{equation}

\textbf{Step 2}: 
Next we establish with the help of (46) an estimate for $\bold {p}$. By Theorem 4.7. we can choose for any $\epsilon>0$, $t>\bar t$ a $\mathbb{P}^{\epsilon}\in\mathcal{P}^f(\bar t,\bar p)$ such that we have
\begin{eqnarray}
\epsilon(t-\bar t)\geq \mathbb{E}_{\mathbb{P}^\epsilon}\left[\int_{\bar t}^{t} \tilde H(s,X^{\bar t,\bar x}_{s},z^{\bar t,\bar x}_s,\bold p_s)ds+W(t,X^{\bar t,\bar x}_{t},\bold p_{t-})-W(\bar t,\bar x,\bar p)\big|\mathcal{F}_{\bar t-}\right].
\end{eqnarray}
Hence by (46) it holds for all $t>\bar t$, such that $(t-\bar t)$ is sufficiently small,
\begin{equation}
\begin{array}{rcl}
\epsilon(t-\bar t)&\geq& \mathbb{E}_{\mathbb{P}^\epsilon}\bigg[\int_{\bar t}^{t} \tilde H(s,X^{\bar t,\bar x}_{s},z_s^{\bar t,\bar x},\bold p_s)ds+\phi(t,X^{\bar t,\bar x}_{t},\bar p)-\phi(\bar t,\bar x,\bar p)\\
&&\ \ \ \ \ \ \ \ \ +\langle\frac{\partial \phi }{\partial p}(t,X^{\bar t,\bar x}_{t},\bar p),\bold p_{t-}-\bar p\rangle+\delta  1_{\{|X^{\bar t,\bar x}_{t}-\bar x|<\eta'\}}|\bold p_{t-}-\bar p|^2\big|\mathcal{F}_{\bar t-}\bigg].
\end{array}
\end{equation}
With the estimate (28) we have for a generic constant $c$
\begin{eqnarray}
\left|\mathbb{E}_{\mathbb{P}^\epsilon}\left[\int_{\bar t}^{t} \tilde H(s,X^{\bar t,\bar x}_{s},z^{\bar t,\bar x}_s,\bold p_s)ds\big|\mathcal{F}_{\bar t-}\right]\right|
\leq c \mathbb{E}_{\mathbb{P}^\epsilon}\left[\int_{\bar t}^{t} (1+|z_s^{\bar t,\bar x}|)ds\big|\mathcal{F}_{\bar t-}\right] \leq  c(t-\bar t)^\frac{1}{2},
\end{eqnarray}
since by Proposition A.3.
\begin{eqnarray*}
\mathbb{E}_{\mathbb{P}^\epsilon}\left[\int_{\bar t}^{t} |z_s^{\bar t,\bar x}|^2ds\big|\mathcal{F}_{\bar t-}\right]\leq c \mathbb{E}_{\mathbb{P}^\epsilon}\left[\int_{\bar t}^{t} |X_s^{\bar t,\bar x}|^2ds\big|\mathcal{F}_{\bar t-}\right]\leq c.
\end{eqnarray*}
Furthermore by It\^o's formula it holds 
\begin{eqnarray}
&&\left|\mathbb{E}_{\mathbb{P}^\epsilon}\left[\phi(t,X^{\bar t,\bar x}_{t},\bar p)-\phi(\bar t,\bar x,\bar p)\big|\mathcal{F}_{\bar t-}\right]\right|\leq c(t-\bar t).
\end{eqnarray}
Next, let $f:[\bar t,t]\times\mathbb{R}^n\rightarrow \mathbb{R}^n$ be a smooth bounded function, with bounded derivatives.  Recall that by assumption $\langle B,\bold p^c\rangle=0$ under any ${\mathbb{P}}\in\mathcal{P}(\bar t, \bar p)$. So since under ${\mathbb{P}^\epsilon}$ the process $\bold p$ is a martingale with $\mathbb{E}_{\mathbb{P}^\epsilon}\left[\bold p_{t-}|\mathcal{F}_{\bar t-}\right]=\bar p$, it holds by It\^o's formula
\begin{eqnarray*}
&&\mathbb{E}_{\mathbb{P}^\epsilon}\left[f_i(t,X^{\bar t,\bar x}_{t})(\bold p_{t-}-\bar p)_i\big|\mathcal{F}_{\bar t-}\right]\\
&&\ \ =\mathbb{E}_{\mathbb{P}^\epsilon}\left[\int_{\bar t}^tf_i(s,X^{\bar t,\bar x}_{s})d(\bold p_s)_i+\int_{\bar t}^t(\bold p_s-\bar p)_idf_i(s,X^{\bar t,\bar x}_{s})+\left[f_i(\cdot,X^{\bar t,\bar x}_{\cdot}),(\bold p_\cdot-\bar p)_i\right]_{t-}\big|\mathcal{F}_{\bar t-}\right]\\
&&\ \ =\mathbb{E}_{\mathbb{P}^\epsilon}\bigg[\int_{\bar t}^t \bigg(\frac{\partial}{\partial t}f_i(s,X^{\bar t,\bar x}_{s})+\langle D_x f_i(s,X^{\bar t,\bar x}_{s}),b(X^{\bar t,\bar x}_{s})\rangle\\
&&\ \ \ \ \ \ \ \ \ \ \ \ \ \ \ \ \ \ +\frac{1}{2}\textnormal{tr}(\sigma\sigma^*(s,X^{\bar t,\bar x}_{s})D^2_xf_i(s,X^{\bar t,\bar x}_{s}))\bigg)( \bold p_s-\bar p)_i ds\big|\mathcal{F}_{\bar t-}\bigg].
\end{eqnarray*}
Hence by (H)
\begin{eqnarray}
&&\left|\mathbb{E}_{\mathbb{P}^\epsilon}\left[\langle\frac{\partial \phi }{\partial p}(t,X^{\bar t,\bar x}_{t},\bar p),\bold p_{t-}-\bar p\rangle\big|\mathcal{F}_{\bar t-}\right]\right| \leq c \mathbb{E}_{\mathbb{P}^\epsilon}\left[\int_{\bar t}^t |\bold p_s-\bar p| ds\big|\mathcal{F}_{\bar t-}\right]\leq c(t-\bar t).
\end{eqnarray}
Furthermore observe that, since $|\bold p_{t-}-\bar p|\leq1$, it holds for $\epsilon'>0$ by Young and H\"older inequality 
\begin{eqnarray*}
\mathbb{E}_{\mathbb{P}^\epsilon}\left[1_{\{|X^{\bar t,\bar x}_{t}-\bar x|<\eta'\}}|\bold p_{t-}-\bar p|^2\big|\mathcal{F}_{\bar t-}\right]&=&\mathbb{E}_{\mathbb{P}^\epsilon}\left[(1-1_{\{|X^{\bar t,\bar x}_{t}-\bar x|\geq\eta'\}})|\bold p_{t-}-\bar p|^2\big|\mathcal{F}_{\bar t-}\right]\\
&\geq& \mathbb{E}_{\mathbb{P}^\epsilon}\left[|\bold p_{t-}-\bar p|^2\big|\mathcal{F}_{\bar t-}\right]-\frac{1}{\eta'} \mathbb{E}_{\mathbb{P}^\epsilon}\left[|X^{\bar t,\bar x}_{t}-\bar x| |\bold p_{t-}-\bar p|^2\big|\mathcal{F}_{\bar t-}\right]\\
&\geq&  \mathbb{E}_{\mathbb{P}^\epsilon}\left[|\bold p_{t-}-\bar p|^2\big|\mathcal{F}_{\bar t-}\right]-\frac{1}{\eta'}  \mathbb{E}_{\mathbb{P}^\epsilon}\left[|X^{\bar t,\bar x}_{t}-\bar x| |\bold p_{t-}-\bar p|\big|\mathcal{F}_{\bar t-}\right]\nonumber\\
&\geq&(1- \frac{\epsilon' }{\eta'} )\mathbb{E}_{\mathbb{P}^\epsilon}\left[|\bold p_{t-}-\bar p|^2\big|\mathcal{F}_{\bar t-}\right]-\frac{1}{4\eta'\epsilon'} \mathbb{E}_{\mathbb{P}^\epsilon}\left[|X^{\bar t,\bar x}_{t}-\bar x|^2\big|\mathcal{F}_{\bar t-}\right],
\end{eqnarray*}
hence 
\begin{eqnarray}
\mathbb{E}_{\mathbb{P}^\epsilon}\left[1_{\{|X^{\bar t,\bar x}_{t}-\bar x|<\eta'\}}|\bold p_{t-}-\bar p|^2\big|\mathcal{F}_{\bar t-}\right]&\geq& (1- \frac{\epsilon' }{\eta'} )\mathbb{E}_{\mathbb{P}^\epsilon}\left[|\bold p_{t-}-\bar p|^2\big|\mathcal{F}_{\bar t-}\right]-\frac{1}{4\eta'\epsilon'} (t-\bar t).
\end{eqnarray}
Choosing $0<\epsilon'<\eta'$ and combining (48) with the estimates (49)-(52) there exists a constant $c$, such that
\begin{eqnarray}
\mathbb{E}_{\mathbb{P}^\epsilon}\left[|\bold p_{t-}-\bar p|^2\big|\mathcal{F}_{\bar t-}\right]\leq c (t-\bar t)^\frac{1}{2}.
\end{eqnarray}
Since $\bold p$ is a martingale, it holds for all $s\in[\bar t,t[$
\begin{eqnarray*}
\mathbb{E}_{\mathbb{P}^\epsilon}\left[|\bold p_s-\bar p|^2\big|\mathcal{F}_{\bar t-}\right]\leq c(t-\bar t)^\frac{1}{2},
\end{eqnarray*}
hence
\begin{eqnarray}
\mathbb{E}_{\mathbb{P}^\epsilon}\left[\int_{\bar t}^{t} |\bold p_s-\bar p|ds\big|\mathcal{F}_{\bar t-}\right]
	\leq (t-\bar t)^\frac{1}{2} \mathbb{E}_{\mathbb{P}^\epsilon}\left[\int_{\bar t}^{t} |\bold p_s-\bar p|^2ds\big|\mathcal{F}_{\bar t-}\right]^\frac{1}{2} \leq  c(t-\bar t)^\frac{5}{4}.
\end{eqnarray}

\textbf{Step 3}: Note that under $\mathbb{P}^\epsilon\in\mathcal{P}^f(\bar t,\bar p)$ the triplet $(Y^{\bar t,\bar x}_s,z^{\bar t,\bar x}_s,N_s)_{s\in[\bar t,T]}$ is given by the unique solution to the BSDE 
\begin{eqnarray*}
Y_s^{\bar t,\bar x}=\langle\bold p_T, g(X_T^{\bar t,\bar x})\rangle+\int_s^T \tilde H(r,X_r^{\bar t,\bar x},z^{\bar t,\bar x}_r,\bold{p}_r)dr-\int_s^Tz^{\bar t,\bar x}_rdB_r-N_T+N_s.
\end{eqnarray*}
To consider an auxiliary BSDE with terminal time $t$ we define as in the standard case (see e.g. \cite{ElK})
\begin{eqnarray*}
G(s,x,p)&=&\frac{\partial \phi}{\partial t}(s,x,p)+\frac{1}{2}\textnormal{tr}(\sigma\sigma^*(s,x)D^2\phi(s,x,p))+\tilde H(t,x,\sigma^* (s,x)D\phi(s,x,p),p)\\
&=&\frac{\partial \phi}{\partial t}(s,x,p)+\frac{1}{2}\textnormal{tr}(\sigma\sigma^*(s,x)D^2\phi(s,x,p))+H(t,x,D\phi(s,x,p),p)
\end{eqnarray*}
and set 
\begin{eqnarray*}
\tilde Y^{\bar t,\bar x}_s&=& Y^{\bar t,\bar x}_s-\phi(s,X^{\bar t, \bar x}_s,\bold p_{s})-\int_s^t G(r,\bar x,\bar p) dr\\
&&\ \ \ \ \ +\sum_{\bar t \leq r\leq s} \left(\phi(r,X^{\bar t,\bar x}_r,\bold p_{r})-\phi(r,X^{\bar t,\bar x}_r,\bold p_{r-})-\langle \frac{\partial}{\partial p}\phi(r,X^{\bar t,\bar x}_r,\bold p_{r-}),\bold p_{r}-\bold p_{r-}\rangle\right)\\
\tilde z^{\bar t,\bar x}_s&=& z^{\bar t,\bar x}_s-\sigma^*(s,X^{\bar t, \bar x}_s)D_x\phi(s,X^{\bar t, \bar x}_s,\bold p_{s}).
\end{eqnarray*}
Then by It\^o's formula the triplet $(\tilde Y^{\bar t,\bar x},\tilde z^{\bar t,\bar x}, N)$ fulfills 
\begin{eqnarray*}
\tilde Y^{\bar t,\bar x}_s&=& Y^{\bar t,\bar x}_{t-}+\int_s^t \bigg(\tilde{H}(r,X^{\bar t,\bar x}_r,\tilde z^{\bar t,\bar x}_r+\sigma^*(r,X^{\bar t, \bar x}_r)D_x\phi(r,X^{\bar t, \bar x}_r,\bold p_{r}), \bold p_r)\bigg)dr\\
&&\ \ \ \ \ -\int_s^t \tilde z^{\bar t,\bar x}_rdB_r- N_{t-}+N_s\\
&&-\phi(t,X^{\bar t, \bar x}_t,\bold p_{t-})+\int_s^t \bigg(\frac{\partial \phi}{\partial t}(r,X^{\bar t, \bar x}_r,\bold p_r)+\frac{1}{2}\textnormal{tr}(\sigma\sigma^*(r,X^{\bar t, \bar x}_r)D^2\phi(r,X^{\bar t, \bar x}_r,\bold p_r))\bigg)dr\\
&&\ \ \ \ \ +\sum_{s \leq r< t } \left(\phi(r,X^{\bar t,\bar x}_r,\bold p_{r})-\phi(r,X^{\bar t,\bar x}_r,\bold p_{r-})-\langle \frac{\partial}{\partial p}\phi(r,X^{\bar t,\bar x}_r,\bold p_{r-}),\bold p_{r}-\bold p_{r-}\rangle\right)\\
&&-\int_s^t G(r,\bar x,\bar p) dr\\
&&+\sum_{\bar t \leq r\leq s} \left(\phi(r,X^{\bar t,\bar x}_r,\bold p_{r})-\phi(r,X^{\bar t,\bar x}_r,\bold p_{r-})-\langle \frac{\partial}{\partial p}\phi(r,X^{\bar t,\bar x}_r,\bold p_{r-}),\bold p_{r}-\bold p_{r-}\rangle\right),
\end{eqnarray*}
hence is on $[\bar t,t[$ the solution to the BSDE
\begin{eqnarray*}
\tilde Y^{\bar t,\bar x}_s&=&\xi+\int_s^t \bigg(\tilde{H}(r,X^{\bar t,\bar x}_r,\tilde z^{\bar t,\bar x}_r+\sigma^*(r,X^{\bar t, \bar x}_r)D_x\phi(r,X^{\bar t, \bar x}_r,\bold p_{r}), \bold p_r)\\
&&\ \ \ \ \ \ \ \ \ \ \ \ \ \ \ \ +\frac{\partial \phi}{\partial t}(r,X^{\bar t, \bar x}_r,\bold p_r)+\frac{1}{2}\textnormal{tr}(\sigma\sigma^*(r,X^{\bar t, \bar x}_r)D^2\phi(r,X^{\bar t, \bar x}_r,\bold p_r))-G(r,\bar x,\bar p)\bigg)dr\\
&&\ \ \ \ -\int_s^t \tilde z^{\bar t,\bar x}_rdB_r-N_{t-}+ N_s
\end{eqnarray*}
with the terminal value
\[ \xi=\bar Y^{\bar t,\bar x}_{t-}-\phi(t,X^{\bar t, \bar x}_t,\bold p_{t-})+\sum_{\bar t \leq r< t} \left(\phi(r,X^{\bar t,\bar x}_r,\bold p_{r})-\phi(r,X^{\bar t,\bar x}_r,\bold p_{r-})-\langle \frac{\partial}{\partial p}\phi(r,X^{\bar t,\bar x}_r,\bold p_{r-}),\bold p_{r}-\bold p_{r-}\rangle\right).
\]
Note that by the strict convexity assumption on $\phi$ it holds ${\mathbb{P}^\epsilon}$-a.s.
\begin{eqnarray}
\sum_{ \bar t\leq r< t} \left(\phi(r,X^{\bar t,\bar x}_r,\bold p_{r})-\phi(r,X^{\bar t,\bar x}_r,\bold p_{r-})-\langle \frac{\partial}{\partial p}\phi(r,X^{\bar t,\bar x}_r,\bold p_{r-}),\bold p_{r}-\bold p_{r-}\rangle\right) \geq0.
\end{eqnarray}
Furthermore by Lemma 4.8. and the choice of  $\phi$ we have $Y^{\bar t,\bar x}_{t-}\geq W(t,X^{\bar t,\bar x}_{t},\bold p_{t-})\geq \phi(t,X^{\bar t,\bar x}_{t},\bold p_{t-})$, hence $\xi\geq0$.\\
Consider now the solution to the BSDE with the same driver but target $0$, i.e.
\begin{equation}
\begin{array}{rcl}
\bar Y^{\bar t,\bar x}_s&=&\int_s^t \big(\tilde{H}(r,X^{\bar t,\bar x}_r,\bar z^{\bar t,\bar x}_r+\sigma^*(r,X^{\bar t, \bar x}_r)D_x\phi(r,X^{\bar t, \bar x}_r,\bold p_{r}), \bold p_r)\\
\ \\
&&\ \ \ \ \ +\frac{\partial \phi}{\partial t}(r,X^{\bar t, \bar x}_r,\bold p_r)+\frac{1}{2}\textnormal{tr}(\sigma\sigma^*(r,X^{\bar t, \bar x}_r)D^2\phi(r,X^{\bar t, \bar x}_r,\bold p_r))-G(r,\bar x,\bar p)\big)dr\\
\ \\
&&\ \ \ \ \ -\int_s^t \bar z^{\bar t,\bar x}_rdB_r-\bar N_{t-}+\bar N_s.
\end{array}
\end{equation}
Note that by Theorem A.4. we have
\begin{eqnarray}
\tilde Y^{\bar t,\bar x}_{\bar t-}\geq  \bar Y^{\bar t,\bar x}_{\bar t-},
\end{eqnarray}
while by Proposition A.3. it holds
\begin{eqnarray}
\mathbb{E}_{\mathbb{P}^\epsilon}\left[\int_{\bar{t}}^t |\bar z^{\bar t,\bar x}_s|^2ds\big|\mathcal{F}_{\bar t-}\right]\leq c \mathbb{E}_{\mathbb{P}^\epsilon}\left[\int_{\bar t}^t|\bar f_s|^2ds\big|\mathcal{F}_{\bar t-}\right]
\end{eqnarray}
with
\begin{eqnarray*}
&& \bar f_s:=\tilde{H}(s,X^{\bar t,\bar x}_s,\sigma^*(s,X^{\bar t, \bar x}_s)D_x\phi(s,X^{\bar t, \bar x}_s,\bold p_s), \bold p_s)\\
&&\ \ \ \ \ \ \ \ \  \ \ +\frac{\partial \phi}{\partial t}(s,X^{\bar t, \bar x}_s,\bold p_s)+\frac{1}{2}\textnormal{tr}(\sigma\sigma^*(s,X^{\bar t, \bar x}_s)D^2\phi(s,X^{\bar t, \bar x}_s,\bold p_s))-G(s,\bar x,\bar p).
\end{eqnarray*}
Because $\tilde H$ is uniformly Lipschitz continuous in $p$ and the derivatives of $\phi$ with respect to $p$ are uniformly bounded, we have
\begin{eqnarray*}
|\bar f_s|
&\leq& \bigg|\tilde{H}(s,X^{\bar t,\bar x}_s,\sigma^*(s,X^{\bar t, \bar x}_s)D_x\phi(s,X^{\bar t, \bar x}_s,\bar p), \bar p)\\
&&\ \ \ \ \ \ \ \ \ \  \ \ \ \ \ +\frac{\partial \phi}{\partial t}(s,X^{\bar t, \bar x}_s,\bar p)+\frac{1}{2}\textnormal{tr}(\sigma\sigma^*(s,X^{\bar t, \bar x}_s)D^2\phi(s,X^{\bar t, \bar x}_s,\bar p))-G(s,\bar x,\bar p)\bigg|\\
&&\ \ \ \ \ +c\left| \bold p_s-\bar p\right|
\end{eqnarray*}
and it holds as in \cite{ElK} by the estimate (54) for all $\epsilon'>0$
\begin{eqnarray*}
\mathbb{E}_{\mathbb{P}^\epsilon}\left[\int_{\bar t}^t|\bar f_s|^2ds\big|\mathcal{F}_{\bar t-}\right]&\leq& \frac{1}{4\epsilon'} (t-\bar t)O(t-\bar t)+\epsilon'c\ \ \mathbb{E}_{\mathbb{P}^\epsilon}\left[\int_{\bar t}^t\left| \bold p_s-\bar p\right|^2ds\big|\mathcal{F}_{\bar t-}\right]\\
&\leq& \frac{1}{4\epsilon'} (t-\bar t)O(t-\bar t)+\epsilon'c (t-\bar t)^\frac{3}{2},
\end{eqnarray*}
where $O(t-\bar t)\rightarrow 0$ as $t\rightarrow\bar t$.
Hence it holds by (58) and Cauchy inequality
\begin{eqnarray}
\mathbb{E}_{\mathbb{P}^\epsilon}\left[\int_{\bar{t}}^t |\bar z^{\bar t,\bar x}_s|ds\big|\mathcal{F}_{\bar t-}\right]
\leq c \left((t-\bar t){O(t-\bar t)}+(t-\bar t)^\frac{5}{4}\right)
\end{eqnarray}
and 
\begin{eqnarray*}
\bar Y^{\bar t,\bar x}_{\bar t-}\geq -c \mathbb{E}_{\mathbb{P}^\epsilon}\left[\int_{\bar t}^t|\bar f_s|ds|\mathcal{F}_{\bar t-}\right]- c \mathbb{E}_{\mathbb{P}^\epsilon}\left[\int_{\bar{t}}^t |\bar z^{\bar t,\bar x}_s|ds|\mathcal{F}_{\bar t-}\right]
\geq - c \left((t-\bar t){O(t-\bar t)}+(t-\bar t)^\frac{5}{4}\right).
\end{eqnarray*}
So by the (57) we have
\begin{eqnarray}
\tilde Y^{\bar t,\bar x}_{\bar t-} \geq -c \mathbb{E}_{\mathbb{P}^\epsilon}\left[\int_{\bar t}^t|\bar f_s|ds|\mathcal{F}_{\bar t-}\right]- c \mathbb{E}_{\mathbb{P}^\epsilon}\left[\int_{\bar{t}}^t |\bar z^{\bar t,\bar x}_s|ds|\mathcal{F}_{\bar t-}\right]
\geq - c \left((t-\bar t){O(t-\bar t)}+(t-\bar t)^\frac{5}{4}\right).
\end{eqnarray}

\textbf{Step 4}: The theorem is proved, if we show $G(\bar t, \bar x, \bar p)\leq 0$.
Note that by definition of $\tilde{Y}^{\bar t,\bar x}$
\begin{equation}
\begin{array}{rcl}
\tilde{Y}^{\bar t,\bar x}_{\bar t-}&=&{Y}^{\bar t,\bar x}_{\bar t-}-\phi(\bar t,\bar x, \bar p) -\int_{\bar t}^t G(r,\bar x,\bar p) dr.
\end{array}
\end{equation}
Since $\phi(\bar t,\bar x, \bar p)=W(\bar t,\bar x, \bar p)$, we have by the choice of  $\mathbb{P}^\epsilon$ and the Dynamic Programming (Theorem 4.7.)
\begin{eqnarray*}
{Y}^{\bar t,\bar x}_{\bar t-}-\phi(\bar t,\bar x, \bar p)
&=&{Y}^{\bar t,\bar x}_{\bar t-}-W(\bar t,\bar x, \bar p)\\
&\leq&{Y}^{\bar t,\bar x}_{\bar t-}-\mathbb{E}_{\mathbb{P}^\epsilon}\left[\int_{\bar t}^t \tilde{H}(s,X^{\bar t,\bar x}_s,z^{\bar t,\bar x}_s, \bold p_s)ds+W(t,X^{\bar t,\bar x}_t, \bold p_{t-})\big|\mathcal{F}_{\bar t-}\right]+\epsilon(t-\bar t)\\
&=&\mathbb{E}_{\mathbb{P}^\epsilon}\left[{Y}^{\bar t,\bar x}_{t-}-W(t,X^{\bar t,\bar x}_t, \bold p_{t-})\big|\mathcal{F}_{\bar t-}\right]+\epsilon(t-\bar t).
\end{eqnarray*}
Recall that by the choice of $\mathbb{P}^\epsilon$ according to Lemma 4.9. it holds
\begin{eqnarray*}
				Y^{\bar t,\bar x}_{t-}-W(t,X_{t}^{\bar t,\bar x},\bold p_{t-})\leq c\epsilon(t-\bar t).
\end{eqnarray*}
Hence
\begin{eqnarray}
{Y}^{\bar t,\bar x}_{\bar t-}-\phi(\bar t,\bar x, \bar p)
\leq c \epsilon (t-\bar t).
\end{eqnarray}
Thus from (61) with (62) we have
\begin{eqnarray*}
\tilde{Y}^{\bar t, \bar x}_{\bar t-}+\int_{\bar t}^t G(r,\bar x,\bar p) dr\leq c \epsilon (t-\bar t)
\end{eqnarray*}
and finally by the estimate (60)
\begin{eqnarray*}
-c \left((t-\bar t){O(t-\bar t)}+(t-\bar t)^\frac{5}{4}\right)+\int_{\bar t}^t G(r,\bar x,\bar p) dr\leq c \epsilon (t-\bar t),
\end{eqnarray*}
hence
\begin{eqnarray*}
\frac{1}{(t-\bar t)}\int_{\bar t}^tG(s,\bar x,\bar p)ds\leq  c\  \left({O(t-\bar t)}+(t-\bar t)^\frac{1}{4}\right)+c \epsilon
\end{eqnarray*}
which implies (40) as $t\downarrow \bar t$ since $\epsilon>0$ can be chosen arbitrary small.
\end{proof}

Thus by Proposition 4.11., 4.12. and comparison for (8) (see \cite{CaRa1}, \cite{Ca}) we now have the following result. 
\begin{thm}
$W$ is the unique viscosity solution to (8).
\end{thm}

Theorem 3.4. follows directly from Theorem 4.13. and the characterization of the value function in Theorem 2.7.

\section{Concluding remarks}
In this paper we have shown an alternative representation of the value function in terms of a minimization of solutions of certain BSDEs over some specific martingale measures. These BSDEs correspond to the dynamics of a stochastic differential game with the beliefs of the uninformed player (modulo a Girsanov transformation) as an additional forward dynamic. We used this to show how to explicitly determine the optimal reaction of the informed player under some rather restrictive assumptions. For a generalization a careful analysis of the optimal measure in the representation of Threorem 3.4. is necessary. In the simpler framework of \cite{CaRa1} the existence of a weak limit $\mathbb{P}^*$ for a minimizing sequence is straightforward using \cite{MeZ}. In our case any limiting procedure needs to take into account the BSDE structure. The question of existence of an optimal measure under which there is a representation by a soltution to a BSDE poses therefore a rather delicate problem, which shall be addressed in a subsequent work.

 \appendix
 
 \section{Results for BSDE on $\mathcal{D}([0,T];\Delta(I))\times\mathcal{C}([0,T];\mathbb{R}^d)$}

Here we give proofs for versions of standard BSDE results adapted to our setting. Let $\Omega:=\mathcal{D}([0,T];\Delta(I))\times\mathcal{C}([0,T];\mathbb{R}^d)$ and $(\Omega,\mathcal{F},(\mathcal{F}_s)_{s\in[t,T]})$ be defined as in section 3.1. We fix a $\mathbb{P}\in\mathcal{P}(t,p)$ and denote $\mathbb{E}_\mathbb{P}[\cdot]=\mathbb{E}[\cdot]$.

Let $\xi\in\mathcal{L}_T^2(\mathbb{P})$, i.e.  $\xi$ is a square integrable $\mathcal{F}_T$-measurable random variable. Let $f:\Omega\times[0,T]\times\mathbb{R}^d\rightarrow\mathbb{R}$ be $\mathcal{P}\otimes\mathcal{B}(\mathbb{R}^d)$ measurable, such that $f(\cdot,0)\in\mathcal{H}^2(\mathbb{P})$ and such that, there exists a constant c, such that $\mathbb{P}\otimes dt$ a.s.
\begin{eqnarray}
|f(\omega,s,z^1)-f(\omega,s,z^2)|\leq c |z^1-z^2|\ \ \ \forall z^1,z^2\in\mathbb{R}^d.
\end{eqnarray}

We consider on $\mathcal{D}([0,T];\Delta(I))\times\mathcal{C}([0,T];\mathbb{R}^d)$ the BSDE
\begin{eqnarray}
Y_s&=&\xi+ \int_s^Tf(r,z_r)ds+\int_s^Tz_rdB_r- (N_T-N_s).
\end{eqnarray}

\begin{thm}
For any fixed $\mathbb{P}\in\mathcal{P}(t,x)$ there exists a solution $(Y,z,N)\in\mathcal{H}^2(\mathbb{P})\times\mathcal{H}^2(\mathbb{P})\times \mathcal{M}^2_0(\mathbb{P})$ to (A.2), such that N is strongly orthogonal to $\mathcal{I}^2(\mathbb{P})$. Furthermore $(Y,z)$ are unique in $\mathcal{H}^2(\mathbb{P})\times\mathcal{H}^2(\mathbb{P})$ and $N\in\mathcal{M}^2_0(\mathbb{P})$ is unique up to indistinguability.
\end{thm}

\begin{rem}
The proof we give is a combination of the proof for the solvability of BSDE given in \cite{ElK} and the Galtchouk-Kunita-Watanabe decomposition (see e.g. \cite{AnS}). For the reader's convenience we recall:\\
By the Galtchouk-Kunita-Watanabe Theorem $\mathcal{I}^{2}(\mathbb{P})$ is a stable subspace of $\mathcal{M}^2_0(\mathbb{P})$ and we have for any $\xi\in\mathcal{L}^2_T(\mathbb{P})$ a decomposition
\begin{eqnarray}
\xi=\mathbb{E}_{\mathbb{P}}[\xi|\mathcal{F}_{0}]+\int_{0}^T \theta_s dB_s+N_T
\end{eqnarray}
with a $\theta\in\mathcal{H}^2(\mathbb{P})$ and a $N\in\mathcal{M}^2_0(\mathbb{P})$ which is strongly orthogonal to $\mathcal{I}^{2}(\mathbb{P})$, i.e. $N\int\theta dB$ is a $\mathbb{P}$ martingale for every $\theta\in\mathcal{H}^{2}(\mathbb{P})$ or equivalently $\langle B,{N}^c \rangle_s=0$ for all $ s \in [0,T]$,  where $N^c$ denotes the continuous part of $N$ (since $B,{N}$ are square integrable see \cite{JaShi} I.4.15).
Moreover this representation is unique up to indistinguishability. 
\end{rem}

\begin{proof}
Let $Y^0\equiv0$, $z^{0}\equiv0$ and define recursively for $n\geq1$ by Galtchouk-Kunita-Watanabe 
\begin{eqnarray}
Y^{n}_s&=&\mathbb{E}\left[\int_{s}^T f(r,z^{n-1}_r)dr+\xi\big|\mathcal{F}_s\right]\\
&=&\xi+\int_s^T  f(r,z^{n-1}_r)dr-\int_s^T z^{n}_r dB_r-(N^{n}_T-N^n_s).
\end{eqnarray}
First observe that  for all $n\in\mathbb{N}$ by induction and Burholder-Davis-Gundy the $\mathcal{F}_T$-measurable random variable $\sup_{s\in[t,T]} |Y^{n}_s|$ is square integrable.\\
Set $\delta Y^{n}=Y^{n}-Y^{n-1}$, $\delta z^{n}= z^{n}- z^{n-1}$, $\delta N^{n}= N^{n}- N^{n-1}$. Then it holds by It\^o's formula 
\begin{eqnarray*}
&&e^{\beta s} (\delta Y^{n}_s)^2=e^{\beta T}(\delta Y^{n}_T)^2-\beta\int_s^T e^{\beta r} (\delta Y^{n}_r)^2 dr+2\int_s^T e^{\beta r} \delta Y^{n}_r \left( f(r,z^{n-1}_r)-f(r,z^{n-2}_r)\right)dr\\
&&\ \ \ -\int_s^T e^{\beta r} |\delta z^{n}_s|^2 dr-\int_s^T e^{\beta r}  d \langle (\delta N^{n})^c\rangle_r\\
&&\ \ \ -\sum_{s\leq r\leq T} e^{\beta r} \left[(\delta Y^{n}_{r-}+\Delta\delta N^{n}_r)^2-(\delta Y^{n}_{r-})^2-2\delta Y^{n}_{r-}\Delta\delta N^{n}_r\right]\\
&&\ \ \ -2\int_s^T e^{\beta r} \delta Y^{n}_r \delta z^{n}_r dB_r-2\int_t^T e^{\beta r} \delta Y^{n}_r d\delta N^{n}_r,
\end{eqnarray*}
where $\Delta\delta N^{n}$ denotes the jumps of $\delta N^{n}$.\\
Since $\sup_{s\in[0,T]} |Y^{n}_s|$ is square integrable, all martingales in above equation are real martingales with expectation zero. Hence we have using the Lipschitz assumption (A.1) and Cauchy inequality for all $\epsilon>0$
\begin{eqnarray*}
&&\mathbb{E}\left[e^{\beta s} (\delta Y^{n}_s)^2\right]+\mathbb{E}\left[\int_s^T e^{\beta r} |\delta z^{n}_r|^2 ds\right]+\mathbb{E}\left[\int_s^T e^{\beta r}  d \langle \delta N^{n}\rangle_r\right]\\
&&=\mathbb{E}\left[e^{\beta T}(\delta Y^{n}_T)^2\right]-\mathbb{E}\left[\beta\int_s^T e^{\beta r} (\delta Y^{n}_r)^2 dr\right]+2\mathbb{E}\left[\int_s^T e^{\beta r} \delta Y^{n}_r \left(f(r,z^{n-1}_r)-f(r,z^{n-2}_r)\right)dr\right]\\
&&\leq \mathbb{E}\left[e^{\beta T}(\delta Y^{n}_T)^2\right]-\mathbb{E}\left[\beta\int_s^T e^{\beta r} (\delta Y^{n}_r)^2 dr\right]-2c\mathbb{E}\left[\int_s^T e^{\beta r} \delta Y^{n}_r  |\delta z^{n-1}_r| dr\right]\\
&&\leq  \mathbb{E}\left[e^{\beta T}(\delta Y^{n}_T)^2\right]+\left(\frac{c}{\epsilon}-\beta\right) \mathbb{E}\left[\int_s^T  e^{\beta r}(\delta Y^{n}_r)^2 dr\right]
+c\epsilon \ \mathbb{E}\left[\int_s^T e^{\beta r} |\delta z^{n-1}_r|^2 dr\right].
\end{eqnarray*}
Thus
\begin{eqnarray*}
&&\mathbb{E}\left[e^{\beta s} (\delta Y^{n}_s)^2\right]+\left(\beta-\frac{c}{\epsilon}\right) \mathbb{E}\left[\int_s^T e^{\beta r}(\delta Y^{n}_r)^2 dr\right]+\mathbb{E}\left[\int_s^T e^{\beta r} |\delta z^{n}_r|^2 dr\right]+\mathbb{E}\left[\int_s^T e^{\beta r}  d \langle \delta N^{n}\rangle_r\right]\\
&&\ \ \ \ \ \ \ \ \ \ \ \ \ \ \ \ \ \ \ \leq  \mathbb{E}\left[e^{\beta T}(\delta Y^{n}_T)^2\right] + c\epsilon \ \mathbb{E}\left[\int_s^T e^{\beta r} |\delta z^{n-1}_r|^2 dr\right].
\end{eqnarray*}
Since by construction $\delta Y^{n}_T=0$ we have choosing $\epsilon<\frac{1}{c}$ and $\beta>c^2$ 
\begin{eqnarray*}
\mathbb{E}\left[\int_s^T e^{\beta r} |\delta z^{n}_r|^2 dr\right] \leq   c\epsilon \ \mathbb{E}\left[\int_s^T e^{\beta r} |\delta z^{n-1}_r|^2 dr\right],
\end{eqnarray*}
hence convergence of $(z^{n})_{n\in\mathbb{N}}$ in the space $\mathcal{H}^2_\beta(\mathbb{P})$ with a weighted norm implying the convergence of $(z^{n})_{n\in\mathbb{N}}$ in $\mathcal{H}^2(\mathbb{P})$ to a process $z$.
It follows then immediately by (A.3), that $(Y^{n})_{n\in\mathbb{N}}$ converges in $\mathcal{H}^2(\mathbb{P})$ and by Galtchouk-Kunita-Watanabe there exist unique (up to indistinguishability) $N\in\mathcal{M}^2_0(\mathbb{P})$, such that the triplet $(Y,z,N)$ solves (A.2). 
\end{proof}

In other words there exists a unique $z\in\mathcal{H}^2(\mathbb{P})$ such that $Y$ can be represented as
\begin{eqnarray}
Y_s=\mathbb{E}\left[\int_{s}^T  f(r,z_r)ds+\xi \big|\mathcal{F}_s\right].
\end{eqnarray}

Furthermore we note that by the very same methods as in the proof to Theorem A.1. we have the following dependence on the data.

\begin{prop}
For $i=1,2$, let $\xi^i\in\mathcal{L}_T^2(\mathbb{P})$. Let $f^i:\Omega\times[0,T]\times\mathbb{R}^d\rightarrow\mathbb{R}$ be two generators for the BSDE (A.2), i.e. $\mathcal{P}\otimes\mathcal{B}(\mathbb{R}^d)$ measurable, $f^i(\cdot,0)\in\mathcal{H}^2(\mathbb{P})$ and $f^i$ are uniformly Lipschitz continuous in z.\\
Let $(Y^{i},z^{i},N^i)\in\mathcal{H}^2(\mathbb{P})\times\mathcal{H}^2(\mathbb{P}) \times \mathcal{M}^2_0(\mathbb{P})$ be the respective solutions. Set $\delta z=z^{1}-z^{2}$ and $\delta \xi =\xi^1-\xi^2$, $\delta f=f^1(\cdot,z^{2}_\cdot)-f^2(\cdot,z^{2}_\cdot)$. Then it holds for any $s\in[0,T]$
\begin{eqnarray}
\mathbb{E}\left[\int_s^T |\delta z_r|^2 dr|\mathcal{F}_s\right]\leq c \left(\mathbb{E}\left[|\delta \xi |^2|\mathcal{F}_s\right]+\mathbb{E}\left[\int_s^T |\delta f_r|^2dr\big|\mathcal{F}_s\right]\right)
\end{eqnarray}
\end{prop}

Also we have the following comparison principle.

\begin{thm}
For $i=1,2$, let $\xi^i\in\mathcal{L}_T^2(\mathbb{P})$. Let $f^i:\Omega\times[0,T]\times\mathbb{R}^d\rightarrow\mathbb{R}$ be two generators for the BSDE (A.2), i.e. $f^i$ is $\mathcal{P}\otimes\mathcal{B}(\mathbb{R}^d)$ measurable, uniformly Lipschitz continuous in $z$ and $f^i(\cdot,0)\in\mathcal{H}^2(\mathbb{P})$.\\
Let $(Y^{i},z^{i},N^i)\in\mathcal{H}^2(\mathbb{P})\times\mathcal{H}^2(\mathbb{P}) \times \mathcal{M}^2_0(\mathbb{P})$ be the respective solutions. Assume
\begin{itemize}
\item[(i)] $\delta \xi= \xi^1-\xi^2\geq 0$ holds $\mathbb{P}$-a.s.
\item[(ii)] $\delta f=f^1(\cdot,z^{2}_\cdot)-f^2(\cdot,z^{2}_\cdot)\geq 0$ holds $\mathbb{P}\otimes dt$-a.s.
\end{itemize}
Then for any time $s\in[0,T]$ it holds $Y^{1}_s-Y^{2}_s\geq 0$ $\mathbb{P}$-a.s.
\end{thm}

\begin{proof}
Set $\delta Y_s=Y^{1}_s-Y^{2}_s$, $\delta z_s=z^{1}_s-z^{2}_s$, $\delta N_s=N^{1}_s-N^{2}_s$. For $(z^{1}_s)_k-(z^{2}_s)_k>0$ set
\begin{eqnarray*}
\Delta^z f_s=\frac{f^1(s,\tilde z^{k-1}_s)-f^1(s,\tilde z^{k}_s)}{(z^{1}_s)_k-(z^{2}_s)_k},
\end{eqnarray*}
where $\tilde z^{k}=((z^{2})_1,\ldots,(z^{2})_k,(z^{1})_{k+1},\ldots,(z^{\bold p,1})_d)$, and $\Delta^z f=0$ else.\\
Then $\delta Y^\bold p$ solves the linear BSDE
\begin{eqnarray}
\delta Y_s&=&\delta \xi_T+\int_s^T\left(\Delta^z f_r\  \delta z_r+\delta f_r \right) dr-\int_s^T\delta z_r dB_r-\delta N_T+\delta N_s
\end{eqnarray}
Since $f^1$ is uniformly Lipschitz continuous in $z$, $\Delta^z f_t$ is bounded. Hence for any $s\in[t,T]$ the stochastic exponentials
\begin{eqnarray*}
\Gamma^s_r=\mathcal{E}\left(\int_s^r  \ \Delta^z f_u\ dB_u\right)\ \ \ \ r\in[s,T]
\end{eqnarray*}
are real positive martingales with expectation $1$ and by Girsanov (see e.g. Theorem III.3.24 \cite{JaShi}) the solution of the linear BSDE (A.8) is given by
\begin{eqnarray*}
\delta Y_s=\mathbb{E}\left[ \delta \xi \ \Gamma^s_T+\int_s^T \Gamma^s_r\ \delta f_r\ dr|\mathcal{F}_s \right].
\end{eqnarray*}
Thus $\delta Y_s\geq 0$ almost surely for any time $s\in[0,T]$.
\end{proof}
\ \\
\ \\
\textit{Acknowledgements} 
I express my gratitude to Rainer Buckdahn, Pierre Cardaliaguet and Catherine Rainer for helpful discussions and valuable comments.
\ \\



\bibliographystyle{model1a-num-names}
\bibliography{<your-bib-database>}



\end{document}